\documentclass[12pt,reqno]{amsart}

\usepackage{fullpage}
\usepackage{comment}
\usepackage{enumerate}
\usepackage{amssymb}
\usepackage{amsthm}
\usepackage[normalem]{ulem}
\usepackage{thmtools}
\usepackage{mathtools}
\usepackage{pgf,tikz}
\usepackage{float}
\usetikzlibrary{arrows}
\usepackage{hyperref}
\usepackage{cite}

\usepackage{color}
\usepackage{transparent}
\usepackage{graphicx}
\usepackage{subcaption}
            
\definecolor{azure(colorwheel)}{rgb}{0.0, 0.5, 1.0}
\definecolor{hanpurple}{rgb}{0.32, 0.09, 0.98}
\definecolor{iris}{rgb}{0.35, 0.31, 0.81}
\definecolor{byzantine}{rgb}{0.74, 0.2, 0.64}
\definecolor{ashgrey}{rgb}{0.7, 0.75, 0.71}
\definecolor{battleshipgrey}{rgb}{0.52, 0.52, 0.51}

\usepackage{xcolor}
\usepackage[normalem]{ulem}

\hypersetup{colorlinks=true,pdfborder=0 0 0,  citecolor  = blue, citebordercolor= {magenta}}
\hypersetup{linkcolor=black}
\makeatletter
\let\reftagform@=\tagform@
\def\tagform@#1{\maketag@@@{(\ignorespaces\textcolor{purple}{#1}\unskip\@@italiccorr)}}
\renewcommand{\eqref}[1]{\textup{\reftagform@{\ref{#1}}}}
\makeatother

\makeatletter

\DeclareUrlCommand\ULurl@@{%
  \def\UrlLeft{\uline\bgroup}%
  \def\UrlRight{\egroup}}
\def\ULurl@#1{\hyper@linkurl{\ULurl@@{#1}}{#1}}
\DeclareRobustCommand*\ULurl{\hyper@normalise\ULurl@}
\makeatother

\def\lessim{\ \lower4pt\hbox{$
		\buildrel{\displaystyle <}\over\sim$}\ }
\def\gessim{\ \lower4pt\hbox{$\buildrel{\displaystyle >}
		\over\sim$}\ }

\usepackage{amsfonts,txfonts,psfrag}

\usepackage[mathscr]{euscript}

\usepackage{mdframed}

\numberwithin{equation}{section}

\makeatletter
\renewcommand{\@secnumfont}{\bfseries}
\makeatother

\newtheorem{thm}{Theorem}[section]

\theoremstyle{definition}

\newtheorem{ass}[thm]{Assumption}

\newcommand{\Crt}{\mathrm{Crt}}
\newcommand{\indi}{\mathbf{1}}

\newtheorem{theorem}{Theorem}
\numberwithin{theorem}{section}

\newtheorem{corollary}[theorem]{Corollary}

\newtheorem{definition}[theorem]{Definition}

\newtheorem{lemma}[theorem]{Lemma}

\newtheorem{proposition}[theorem]{Proposition}
\newtheorem{remark}[theorem]{Remark}

\title{Sharp complexity asymptotics and topological trivialization for the $(p,k)$ spiked tensor model}
\author{Antonio Auffinger}
\address{Northwestern University}
\email{tuca@northwestern.edu}
\author{Gerard Ben Arous}
\address{New York University}
\email{benarous@cims.nyu.edu}
\author{Zhehua Li}
\address{Northwestern University}
\email{zhehua.li@northwestern.edu}
\begin{document}

\begin{abstract} We provide $O(1)$ asymptotics for the average number of deep minima of the $(p,k)$ spiked tensor model. We also derive an explicit formula for the limiting ground state energy on the $N$-dimensional sphere, similar to the work of Jagannath-Lopatto-Miolane\cite{JLM18}. Moreover, when the signal to noise ratio is large enough, the expected number of deep minima is asymptotically finite as $N$ tends to infinity and we determine its limit as the signal-to-noise ratio diverges. 
\end{abstract}
\maketitle
%

\section{Introduction}

Large dimensional rough landscapes play a central role in many different fields of science. Scientists very often face the question ``Given a function in many variables, how does one obtain significant statistical properties that discern noise to relevant data?" Relevant quantities, for instance, are the number of local minima at a given energy, the value of the absolute minimum, the number of saddles and their geometries. 

In this paper, we study one example of such landscapes, the spherical pure $p$-spin in the presence of a non-linear signal. Precisely, let $S^{N-1}\left(\sqrt{N}\right) = \{ \sigma \in \mathbb R^N : \sum_{i=1}^N \sigma_i^2 = N\}$ be the $N$-sphere of radius $\sqrt{N}$ and fix $\textbf{v}_0 \in S^{N-1}\left(\sqrt{N}\right)$. Given integers $p, k \geq 1$, $\lambda \in \mathbb R$, let
\begin{equation}
H_{N}(\sigma)=\frac{1}{N^{\frac{p-1}{2}}}\sum_{1\leq i_1,i_2,\dots,i_p\leq N}J_{i_1,i_2,\dots,i_p}\sigma_{i_1}\sigma_{i_2}\cdots \sigma_{i_p}-\frac{\lambda N}{k} \left(\frac{\sigma\cdot\textbf{v}_0}{N}\right)^{k},
\label{1}\end{equation}  
where $\sigma = (\sigma_1, \ldots, \sigma_N) \in  S^{N-1}\left(\sqrt{N}\right)$ and $(J_{i_1,i_2,\dots,i_p})_{1\leq i_1, \ldots, i_p \leq N}$ are independent standard gaussian random variables. We call $H_{N}$ the Hamiltonian of the $(p,k)$ spiked tensor model.


Without loss of generality, we refer to the direction of $\textbf{v}_0$ as the North Pole of the model and we let  
\begin{align*}
m(\sigma)=\sigma\cdot \textbf{v}_0/N \in [-1,1]
\end{align*} 
be the overlap of $\sigma$ with the signal $\textbf{v}_0$.
The aim of this paper is to investigate the landscape of the random function $H_{N}$ around its ground state energy
\begin{align*}
	L_N:=\min_{\sigma \in S^{N-1}\left(\sqrt{N}\right)} H_{N}(\sigma)
\end{align*} 
and the overlap between its ground state and the signal 
\begin{align*}
	m_N:=\left({\arg\min}_{\sigma \in S^{N-1}\left(\sqrt{N}\right)} H_{N}(\sigma)\right)\cdot \textbf{v}_0.
\end{align*} 
 For each $\lambda>0$, define
\begin{align}\label{eq:lowhighlatitude}
m_\lambda:=\min\left\{1, \left(\frac{\left(p-2\right)\sqrt{p}}{\lambda \sqrt{p-1}}\right)^{\frac{1}{k}}\right\}.
\end{align}
As illustrated in the transformative work of Ros et al. \cite{BBCR18} (see also Sections 2.3 and 2.4 in \cite{BMMN18}),  in the ``low-latitude'' region, $|m|\leq m_{\lambda}$, $H_N$ has a rugged energy landscape, with exponentially many critical values in $N$, resembling the spherical $p$-spin spin glass models \cite{AB13} while, in the ``high-latitude'' region, $|m|\geq m_{\lambda}$, it resembles a convex potential. The study of phase transitions in the topology of level sets of $H_N$, and limit theorems for $m_N$ and  $L_N$ have drawn a lot of attention recently, see for instance \cite{chen2019}, \cite{JLM18}, \cite{Montanari:2014}, \cite{perry2020}.
\begin{figure}[h]
	\centering
	\includegraphics[width=0.4\textwidth]{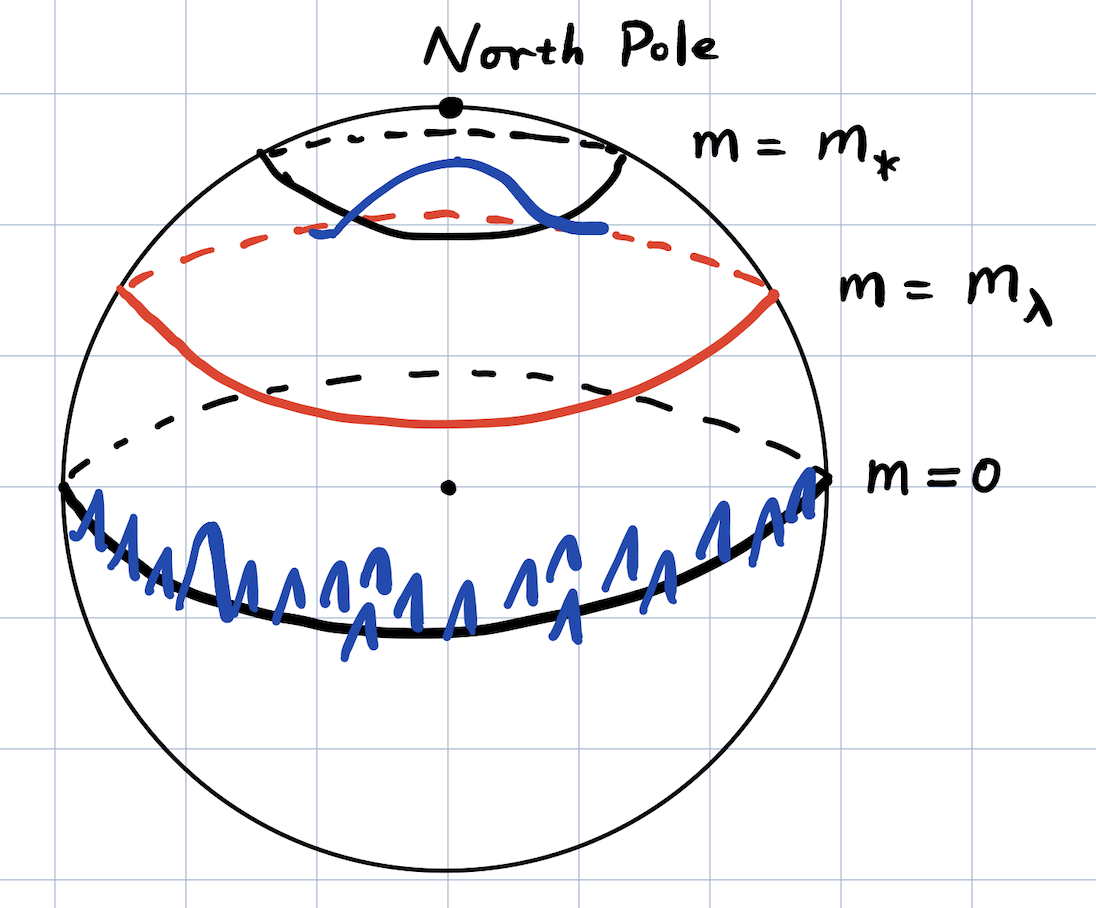}
	\caption{The landscape of $H_N\left(\sigma\right)$ on $S^{N-1}$. $\textbf{v}_0$ is the North Pole, $m=\left<\sigma, \textbf{v}_0\right>/N$ . The spikes around the equator represent numerous local maxima (minima) that are possibly exponential in $N$ in the ``low-latitude'' region $|m| \leq m_{\lambda}$. When $m\geq m_\lambda$, there are only a few critical points on a parallel $m=m_*$.   }
	\label{fig:landscape}
\end{figure}

Here, we focus on providing a better understanding of the model in the presence of a strong signal, that is, when $\lambda$ in  \eqref{1} is large.
In this case, low energy level sets of Hamiltonian $H_{N}$ will go through a phenomena called ``topology trivialization'', a term pioneered by Fyodorov and Le Doussal \cite{Fyodorov14}, and  discussed  in Fyodorov's remarkable work \cite{Fyodorov13, Fyodorov16}.
In short, for $\lambda$ large, one does not expect exponentially many critical values of $H_{N}$ with energy near the ground state energy $L_{N}$. Our first result computes the sharp asymptotics of the average number of critical values and shows that for $\lambda$ sufficiently large, they remain of constant order, not diverging with $N$. In this regime, all critical values are, with probability going to one, local minima. 

Let $\nabla$ and $\nabla^2$ denote the spherical gradient and Hessian with respect to the standard metric on $S^{N-1}\left(\sqrt{N}\right)$. For open sets $M \subseteq [-1,1]$, and $E  \subseteq \mathbb{R}$, we denote the total number of critical points of $H_N$ that have overlap with $\textbf{v}_0$ in $M$ and whose critical values are in $NE$ by
\[\Crt_{N}\left(M,E\right):=\sum_{\sigma\in S^{N-1}\left(\sqrt{N}\right),\nabla H_{N}\left(\sigma\right)=0}\indi_{\left\{\sigma\cdot\textbf{v}_0/N\in M\right\}}\cdot \indi_{\left\{H_N\left(\sigma\right)/N\in E\right\}}\]
and the corresponding number of critical points of index $l = 0, \ldots, N-1$ by
\[\Crt_{N,l}\left(M,E\right):=\sum_{\sigma\in S^{N-1}\left(\sqrt{N}\right),\nabla H_{N}\left(\sigma\right)=0}\indi_{\left\{\sigma\cdot\textbf{v}_0/N\in M\right\}}\cdot \indi_{\left\{H_N\left(\sigma\right)/N\in E\right\}}\indi_{\left\{i\left(\nabla^2 H_N\right)=l\right\}}.\]
Here, the index $i\left(\cdot \right)$ is the number of negative eigenvalues of the corresponding matrix. When $l=0$, $Ctr_{N,0}\left(M,E\right)$ counts the number of local minima that have overlap with $\textbf{v}_0$ in $M$ and whose critical values are in $NE$. Our first main result is the following.

\begin{theorem}\label{thm1}
	Let $M$ be an open interval of $\left(-1,1\right)$ and $E$ be a bounded open interval on $\mathbb{R}$ such that  
	\begin{equation}\label{cond:E}
	\sup E<-2\sqrt{\frac{p-1}{p}}-\left|\frac{1}{p}-\frac{1}{k}\right|\lambda.
	\end{equation}
There exists $c>0$ such that for any $\lambda\geq c$ 
	 there exists a constant $C=C\left(\lambda,p,k\right)$ that does not depend on $M$ and $E$ such that
	\[\lim_{N\to \infty}\mathbb{E}\left[\Crt_{N,0}\left(M,E\right)\right]=\lim_{N\to \infty}\mathbb{E}\left[\Crt_{N}\left(M,E\right)\right]=\begin{cases}
		C>0 & \text{ if }x_*\left(\lambda\right)\in E \text{ and }m_*\left(\lambda\right)\in M, \\
		0 & \text{otherwise}.\\
	\end{cases}  \]
\end{theorem}

The constant $C$ is explicit and  we can further consider its asymptotics when $\lambda\to \infty$.  Let $m_*:=m_*\left(\lambda\right)$ be the largest solution of 
\begin{equation}\label{eq:minima}
\frac{\lambda m^k}{\sqrt{p}}=\frac{m^2}{\sqrt{1-m^2}},
\end{equation}
on $\left(0,1\right]$. Such $m_*$ exists when $\lambda\geq \lambda^{\left(1\right)}\left(p,k\right)=\begin{cases}
	0, & k=1,2 \\
	\sqrt{p\frac{\left(k-1\right)^{k-1}}{\left(k-2\right)^{k-2}}}, & k>2.\\
\end{cases}$.

\begin{theorem}\label{Prop_lamda}
Let $C\left(\lambda,p,k\right)>0$ be the constant given in Theorem \ref{thm1}. Then for any $p\geq 3$ and $k\geq 1$,
\begin{equation*}
\lim_{\lambda\to \infty}C\left(\lambda,p,k\right)=1.
\end{equation*}
\end{theorem}	
Theorem \ref{Prop_lamda} confirms the existence of the trivialization phase for the $(p,k)$ spiked tensor model.  It is believed that as $\lambda\to \infty$, the deterministic potential becomes stronger and the landscape should approach a convex potential with a unique minimum located exactly at the signal vector $\textbf{v}_0$, see \cite{BBCR18}. Closest to our setting is the recent nice work of Belius-et-al.\cite{belius21} which deals with the mean number of critical points for mixed spherical spin glass models with an external field. 

Theorems \ref{thm1} and \ref{Prop_lamda} will follow from the main technical contribution of this paper, which is the derivation of $\mathcal{O}\left(1\right)$ asymptotics of $\mathbb E \left[\Crt_{N}\left(M,E\right)\right]$ in the large $N$ limit.  Exponential asymptotics of  $\mathbb E\left[\Crt_{N}\left(M,E\right)\right] $ were determined by \cite{BMMN18} in the case $k\neq p$. Define $\tilde{S}_{p,k}: \left(-1,1\right)\times  \left(-\infty,-\sqrt{2}\right)\to \mathbb R$ as
	\begin{align}
\tilde{S}_{p,k}\left(m,y\right):=&
\frac{1}{2}\log\left(\left(1-m^2\right)\left(p-1\right)\right)+\frac{2-p}{2p}y
^2-\frac{\lambda m^k}{p}\sqrt{\frac{2\left(p-1\right)}{p}}y-\frac{\lambda^2m^{2k-2}}{2p^2}\left(p+\left(1-p\right)m^2\right)\nonumber\\
&-I_1\left(-y\right),
\end{align}
where 
\begin{align}
	y=y\left(x,m\right):=\frac{px-\left(1-p/k\right)\lambda m^k}{\sqrt{2p\left(p-1\right)}},\label{y}
\end{align}
and
\[I_1\left(z\right)=\int_{\sqrt{2}}^{z}\sqrt{t^2-2}dt \text{  for } z\geq \sqrt{2}, \quad I_1\left(z\right) = \infty \text{ for }z < \sqrt{2} .\]

The next two results do not require any assumptions on $\lambda$.
\begin{theorem}\label{THM1}
	Let $M$ be an open interval of $\left(-1,1\right)$ such that $\bar{M}\subset \left(-1,1\right)$ and $E$ be a bounded open interval on $\mathbb{R}$ such that  $\sup E<-2\sqrt{\frac{p-1}{p}}-\left|\frac{1}{p}-\frac{1}{k}\right|\lambda$,
	then as $N\to \infty$,
	\begin{align}
	\mathbb{E}\left[Crt_{N,0}\left(M,E\right)\right]&=\frac{\sqrt{2}h\left(y_o\right)\left(\sqrt{p}\left(1-m_o^2\right)^{-\frac{3}{2}}-\lambda\left(k-1\right)m_o^{k-2}J\left(m_o,y_o\right)\right)}{\left(\sqrt{y_o^2-2}-y_o\right)p\sqrt{\left|\partial_{yy}\tilde{S}_{p,k}\left(m_o,y_o\right)g^{\prime\prime}\left(m_o\right)\right|}}e^{N\tilde{S}_{p,k}\left(m_o,y_o\right)}\left(1+o\left(N\right)\right),\label{eq:sharp}
	\end{align}
	where
	\begin{align*}
	h\left(y\right)=\left|\frac{y-\sqrt{2}}{y+\sqrt{2}}\right|^{\frac{1}{4}}+\left|\frac{y+\sqrt{2}}{y-\sqrt{2}}\right|^{\frac{1}{4}},
	\end{align*}
	\begin{align*}
	J\left(m,y\right)=\exp\left(-\left(\frac{\lambda^2}{2p^2}m^{2k-2}\left(p\left(1-m^2\right)+m^2\right)+\frac{\lambda m^ky}{2p}\sqrt{\frac{2\left(p-1\right)}{p}}\right)\right),
	\end{align*}
	\begin{align*}
	\tilde{E}_m:=\sqrt{\frac{p}{2\left(p-1\right)}}\left(E-\lambda m^k\left(\frac{1}{p}-\frac{1}{k}\right)\right),\forall m\in M,
\end{align*}
\begin{align*}
	y_o:=y_o\left(m_o\right), y_o\left(m\right)={\arg \max}_{y\in \bar{\tilde{E}}_m} \tilde{S}_{p,k}\left(m,y\right),
\end{align*}
and
\begin{align*}
	 m_o:={\arg\max}_{m\in \bar{M}}g\left(m\right),g\left(m\right)=\tilde{S}_{p,k}\left(m,y_o\left(m\right)\right).
\end{align*}
\end{theorem}

Theorem \ref{THM1} naturally leads to the following corollary.
\begin{corollary}\label{Col1}
	Let $M$ and $E$ be the same as in Theorem \ref{THM1}, then
	\begin{align}
	\lim_{N\to \infty}\frac{1}{N}\log \mathbb{E}\left[Crt_{N,0}\left(M,E\right)\right]=\sup_{m\in \bar{M}}\sup_{y\in \tilde{E}_m}\tilde{S}_{p,k}\left(m,y\right).
	\end{align}
\end{corollary}

\begin{remark}
The function $S_{p,k}\left(m,x\right):= \tilde S_{p,k}(m,y(x,m))$ describes the exponential behavior of $\mathbb{E}\left[\Crt_{N}\left(M,E\right)\right]$ with respect to the dimension $N$ and it is  called the annealed complexity, a function of $m\in \left[-1,1\right]$ and $x\in \mathbb{R}$ such that for any Borel sets $M\subset \left[-1,1\right]$ and $E\subset \mathbb{R}$, 
\begin{align}
	\sup_{m\in M^\mathrm{o},x\in E^\mathrm{o}}S_{p,k,0}\left(m,x\right)&\leq \liminf_{N\to \infty}\frac{1}{N}\log\mathbb{E}\left[\Crt_{N,0}\left(M,E\right)\right]\\\nonumber&\leq \limsup_{N\to \infty}\frac{1}{N}\mathbb{E}\left[\Crt_{N,0}\left(M,E\right)\right]\leq \sup_{m\in \bar{M},x\in \bar{E}}S_{p,k,0}\left(m,x\right).
\end{align}
\end{remark}

\begin{remark} It was discovered in the paper of Ros et al. \cite{BBCR18} that near the signal $v_{0}$ and as $\lambda$ grows, the annealed complexity changes from positive to zero. This transition, named topology trivialization, has been observed and studied in various models of statistical physics and high dimensional optimization, the reader is invited to look at the works of Fyodorov \cite{Fyodorov13, Fyodorov16}, Fyodorov, Le Doussal\cite{Fyodorov14} and Belius et al.\cite{belius21}. The threshold in $\lambda$ in which such transition occurs  is referred to as the trivialization threshold. 
\end{remark}

The $\mathcal{O}\left(1\right)$ asymptotics of Theorem \ref{THM1} also allow us to study the ground state energy in the trivialization region. Our main result in this direction is the following.

\begin{theorem}\label{THM_GSE} For any integers $p\geq 3$ and $k\geq 1$, there exists $c>0$ such that for $\lambda>c$ the following hold:	
\begin{equation}\label{eq-gse22}
	\lim_{N\to \infty} \frac{1}{N}m_N =m_*\left(\lambda\right)  \quad \text{almost surely},
	\end{equation}
	and 
	\begin{equation}\label{eq-gse}
\lim_{N\to \infty} \frac{1}{N}L_N= -\frac{\lambda m_*^k\left(\lambda\right)}{k}-\sqrt{p\left(1-m_*^2\left(\lambda\right)\right)} \quad \text{almost surely}.
	\end{equation}
\end{theorem}

The above theorem was first conjectured and proposed in the article of Ros et al. \cite{BBCR18}, where the authors studied the number of local minima of $H_N$ via a replica theoretic approach. The above formulas are not expected to be true when $\lambda$ is small (see \cite{Gillin00} and \cite{BBCR18} and Remark \ref{rem:TT}) for any choices of $(p,k)$. In the case $p=k$, the $(p,k)$ spiked tensor model has a log-likelihood interpretation as tensor PCA. This interpretation  was used by Jagannath-Lopatto-Miolane to derive asymptotic formulas for the ground state energy for all values of $\lambda$. Theorem \ref{THM_GSE} above is an extension of Theorem 1.2 in \cite{JLM18} for $p\neq q$ and $\lambda$ sufficiently large, although the method of the proof is different.

\begin{remark} \label{rem:TT}(Trivialization threshold). Recall $m_{\lambda}$ from \eqref{eq:lowhighlatitude}. Let 
\[\lambda^{\left(2\right)}\left(p,k\right)=\inf \bigg\{\lambda\geq \lambda^{\left(1\right)}\left(p,k\right):m_*\left(\lambda\right)\geq m_{\lambda} \bigg\}\]
and
\begin{align}\label{eq:tri-threshold}
		\lambda_{tr}=\inf \bigg\{&\lambda\geq \lambda^{\left(2\right)}\left(p,k\right):\sup_{0\leq m\leq m_{\lambda}}S_{p,k} \left(m,x_*\left(\lambda\right)\right)\leq 0 
	\text{ and } \\
		 &S_{p,k}\left(m,x_*\left(\lambda\right)\right)1_{m\in \left[0,m_\lambda\right]} \text{ is a decreasing function of } \lambda  \text{ on } [\lambda^{\left(2\right)}\left(p,k\right),\infty)\label{eq:tri-threshold2}
\bigg\}.
\end{align} 
Our proof of Theorem \ref{THM_GSE} shows that \eqref{eq-gse22} and \eqref{eq-gse} hold for all $\lambda > \lambda_{tr}$. We expect that condition \eqref{eq:tri-threshold} implies \eqref{eq:tri-threshold2} and that \eqref{eq-gse22} and \eqref{eq-gse} fail for $\lambda<\lambda_{tr}$. Figure \ref{fig:monotone} below shows a plot of the annealed complexity for various values of $\lambda$.

For the spiked tensor model $\left(p=k>2\right)$, it has been shown in \cite{BMMN18} and \cite{JLM18} that $\lambda^{\left(1\right)}=\lambda^{\left(2\right)}<\lambda_{tr}$. For the general case, we show in Lemma \ref{lem.compare} that $\lambda^{\left(1\right)}=\lambda^{\left(2\right)}$ if and only if $p\leq k$. However, $\lambda^{\left(2\right)}$ and $\lambda_{tr}$ can only be compared numerically. More details on the existence and values of $\lambda^{\left(i\right)}\left(p,k\right), i=1,2$ and $\lambda_{tr}$ can be found in Lemma \ref{lem.UniqueSol}, \ref{lem.highlatitude}, Proposition \ref{prop.low-latitude-complexity} and Figure \ref{fig:monotone}.

\end{remark}

\begin{figure}[h]
	\centering
	\begin{subfigure}[b]{0.49\textwidth}
		\includegraphics[width=\textwidth]{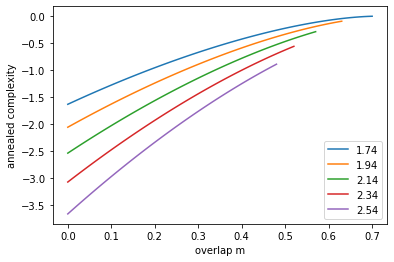}
		\caption{p=3,k=1}
		\label{fig:gull}
	\end{subfigure}
	~ 
	\begin{subfigure}[b]{0.49\textwidth}
		\includegraphics[width=\textwidth]{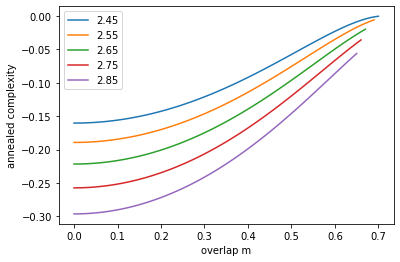}
		\caption{p=3,k=2}
		\label{fig:tiger}
	\end{subfigure}

	~ 
	\begin{subfigure}[b]{0.49\textwidth}
		\includegraphics[width=\textwidth]{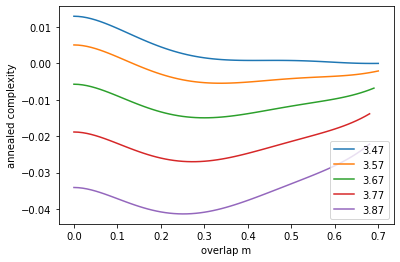}
		\caption{p=3,k=3}
		\label{fig:mouse}
	\end{subfigure}
	\begin{subfigure}[b]{0.49\textwidth}
		\includegraphics[width=\textwidth]{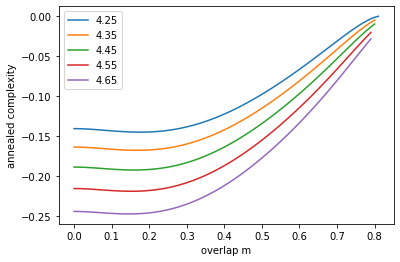}
		\caption{p=4,k=3}
		\label{fig:}
	\end{subfigure}
\caption{$S_{p,k}\left(m,x_*\left(\lambda\right)\right)1_{m\in \left[0,m_\lambda\right]}$ with different values of $p$, $k$ and $\lambda$. The numbers in the legends are values of $\lambda$. In each of the subfigures and for each $m\in \left[0,1\right]$, the values of $S_{p,k}\left(m,x_*\left(\lambda\right)\right)1_{m\in \left[0,m_\lambda\right]}$ decreases as $\lambda$ increases. For $p=3,k=1$, $\lambda^{\left(1\right)}=0, \lambda^{\left(2\right)}=\lambda_{tr}=1.732$.  For $p=3,k=2$, $\lambda^{\left(1\right)}=0, \lambda^{\left(2\right)}=\lambda_{tr}=2.449$.  For $p=3,k=3$, $\lambda^{\left(1\right)}=\lambda^{\left(2\right)}=3.464, \lambda_{tr}=3.619$.  For $p=4,k=3$, $\lambda^{\left(1\right)}=4, \lambda^{\left(2\right)}=\lambda_{tr}=4.243$.}\label{fig:monotone}
\end{figure}


We finish this introduction mentioning a few related results and a brief description of the rest of the paper.  First, the study of models such as the $(p,k)$ spiked tensor along the direction of high dimensional statistical inference was initiated by Montanari-Richard \cite{Montanari:2014}. For the readers who are particularly interested in Tensor PCA and spiked matrix-tensor model, we refer the reader to papers \cite{Montanari:2014, JLM18, BBCR18, Mannellil19a, Mannellil19b,Mannellil20} and the references therein. A prototypical inference model called spiked matrix-tensor model which is closely related to the case of $k=2$ and $p\geq 3$ was extensively studied in \cite{Mannellil19a,Mannellil19b,Mannellil20}. In a recent paper by Maillard-Ben Arous-Biroli \cite{maillard20a}, the complexity study (using the replicated Kac-Rice approach) is extended to current machine learning models like random generalized linear models and neural networks.

In Section \ref{sec.asymptotic} we prove Theorem \ref{THM1}.We first show that the deep minima dominate the total number of critical points in Proposition \ref{prop_dom}. This result allows us to transform the problem of computing the mean number of deep minima into a problem of computing the mean Euler characteristic of  level set for which we could use tools from random matrix theory to compute the characteristic polynomial of a deformed Gaussian Orthogonal Ensemble (GOE).

 In Section \ref{sec.dm} we prove Theorems \ref{thm1} and \ref{Prop_lamda} where we study the mean number of deep minima (minima near the bottom of the energy landscape) and its asymptotic as $\lambda\to \infty$. 

In Section \ref{sec.GSE} we analyse the ground state energy and prove Theorem \ref{THM_GSE}. We first provide in Proposition \ref{Prop_UB} an upper bound of the ground state energy by restricting to energies with fixed latitude $m$, a method that was used \cite{JLM18} in the case of $k=p$. A matching lower bound is given in Proposition \ref{Prop_LB} by exploring the supremum of the annealed complexity near the bottom of the energy landscape.

\color{black}

\section{Proof of Theorem \ref{THM1}}\label{sec.asymptotic}
The (normalized) GOE of size $N$ (denoted by $W_N$) is a real symmetric random matrix $\left(W_{ij}\right)_{1\leq i,j\leq N}$ such that $\left\{W_{ij}\right\}_{1\leq i\leq j\leq N}$ are independent zero mean  normal random variables with $\mathbb{E}\left[W^2_{ij}\right]=\frac{1}{N}$ and $\mathbb{E}\left[W^2_{ii}\right]=\frac{2}{N}$. 

We will work with the rescaled Hamiltonian $f$ on the unit sphere $S^{N-1}$:
\begin{align}
f\left(\sigma\right):=\frac{H_N\left(\sqrt{N}\sigma\right)}{\sqrt{N}}=-\sum_{1\leq i_1,i_2,\dots,i_p\leq N}J_{i_1,i_2,\dots,i_p}\sigma_{i_1}\sigma_{i_2}\cdots\sigma_{i_p}-\frac{\lambda \sqrt{N}}{k}\left<\sigma,\hat{\textbf{v}}_0\right>^k\label{res H}
\end{align} 
where $\hat{\textbf{v}}_0:=\textbf{v}_0\in S^{N-1}$. Then
	\[\Crt_{N}\left(M,E\right)=\sum_{\sigma\in S^{N-1},\nabla f\left(\sigma\right)=0}\indi_{\left\{\sigma\cdot\hat{\textbf{v}}_0\in M\right\}}\cdot \indi_{\left\{f\left(\sigma\right)/\sqrt{N}\in E\right\}}\]
and the corresponding number of critical points of index $l = 0, \ldots, N-1$ by
\[\Crt_{N,l}\left(M,E\right)=\sum_{\sigma\in S^{N-1},\nabla f\left(\sigma\right)=0}\indi_{\left\{\sigma\cdot\hat{\textbf{v}}_0\in M\right\}}\cdot \indi_{\left\{f\left(\sigma\right)/\sqrt{N}\in E\right\}}\indi_{\left\{i\left(\nabla^2 f\right)=l\right\}}.\]

\begin{proposition}\label{prop_dom}
	Let $M$ and $E$ be the same as in Theorem \ref{thm1}, then for any $l\geq 1$, 
	\begin{equation}
	\limsup_{N\to \infty}\frac{1}{N}\log\mathbb{E}\left[Crt_{N,l}\left(M,E\right)\right]<\sup_{m\in \bar{M},x\in \bar{E}}S_{p,k}\left(x,m\right).\label{eq.dom}
	\end{equation}
\end{proposition}
	We postpone the proof of this Proposition to the end of this section. We now show how to prove Theorem \ref{THM1}.

\begin{proof}[Proof of Theorem \ref{THM1}]
	Since $f$ is a Morse function almost surely, let \[\mathcal{S}_N\left(M,E\right):=\left\{\sigma\in S^{N-1}: f\left(\sigma\right)\in \sqrt{N}E, \sigma\cdot \hat{\textbf{v}}_0\in M\right\},\] then its Euler characteristic $\phi\left(\mathcal{S}_N\left(M,E\right)\right)$ can be computed in terms of the numbers of critical points as below,
	\[\phi\left(\mathcal{S}_N\left(M,E\right)\right)=\sum_{l=0}^{N-1}\left(-1\right)^{l+2}Crt_{N,l}\left(M,E\right).\]
	Using Proposition \ref{prop_dom}, we have as $N\to \infty$,
	\begin{equation}
	\mathbb{E}\left[Crt_{N,0}\left(M,E\right)\right]\sim \mathbb{E}\left[Crt_{N}\left(M,E\right)\right]\sim \mathbb{E}\left[\phi\left(\mathcal{S}_N\left(M,E\right)\right)\right].
	\end{equation}
	Therefore  it suffices to compute the asymptotic of the mean Euler characteristic $\mathbb{E}\left[\phi\left(\mathcal{S}_N\left(M,E\right)\right)\right]$. Applying formula 12.4.4 in \cite{AT2007} (see also Eq. $\left(\text{6.22}\right)$ in \cite{AJW09}) , we have
	\begin{equation}
	\mathbb{E}\left[\phi\left(\mathcal{S}_N\left(M,E\right)\right)\right]=\int_{\sigma\cdot v_0\in M}\mathbb{E}\left[ \det \nabla^2 f\left(\sigma\right) \cdot 1_{\left\{f\in \sqrt{N}E\right\}}\mid \nabla f=0 \right]\phi_{\nabla f\left(\sigma\right)}\left(0\right)d\sigma, 
	\end{equation}
	where $\phi_{\nabla f\left(\sigma\right)}\left(0\right)$ is the density of $\nabla f\left(\sigma\right)$ at $0$. 
	
	Let $\omega_{N-2}=\frac{\left(N-1\right)\pi^{\frac{N-1}{2}}}{\Gamma\left(\frac{N+1}{2}\right)}$ be the surface area of $N-2$ dimensional unit sphere, using the data in Lemma \ref{lem_1st dist}, we get
	\begin{align*}
	&\int_{\sigma\cdot v_0\in M}\mathbb{E}\left[ \det \nabla^2 f\left(\sigma\right) \cdot 1_{\left\{f\in \sqrt{N}E\right\}}\mid \nabla f=0 \right]\phi_{\nabla f\left(\sigma\right)}\left(0\right)d\sigma\\
	&=\frac{\omega_{N-2}\sqrt{N}}{\left(2\pi\right)^{\frac{N}{2}}p^{\frac{N-1}{2}}}\int_M\int_E\left(1-m^2\right)^{\frac{N-3}{2}}\exp\left({-\frac{N}{2}\left(\lambda^2m^{2k-2}\left(1-m^2\right)/p+\left(x+\lambda m^k/k\right)^2\right)}\right)G_N\left(x,m\right)dxdm,
	\end{align*}
	where
	\begin{align}\label{eq:decomp1}
	&G_N\left(x,m\right)
	=\left(2\left(N-1\right)p\left(p-1\right)\right)^{\frac{N-1}{2}}\mathbb{E}\left[\det\left(W_{N-1}-\frac{\sqrt{N}}{\sqrt{N-1}}\theta e_{N-1}e^T_{N-1}-\frac{\sqrt{N}}{\sqrt{N-1}}yI_{N-1}\right)\right],
	\end{align}
	\begin{align}\label{eq:decomp2}
	\theta=\theta\left(m\right):=\frac{\lambda\left(k-1\right)m^{k-2}\left(1-m^2\right)}{\sqrt{2p\left(p-1\right)}},
	\end{align}
	and
	\begin{align}
	y=y\left(x,m\right):=\sqrt{\frac{p}{2\left(p-1\right)}}\left(x-\left(1/p-1/k\right)\lambda m^k\right).\label{y-x}
	\end{align}

	Using Lemma \ref{lem_det hermite} and \ref{lem GOE shift+rank1}, we can express $G$ using Hermite polynomials (see definition \ref{def_H}),
	\[G_N\left(x,m\right)=\left(-1\right)^{N-1}\left(p\left(p-1\right)/2\right)^{\frac{N-1}{2}}\left(h_{N-1}\left(\sqrt{N}y\right)+2\sqrt{N}\theta h_{N-2}\left(\sqrt{N}y\right)\right).\]
	It follows that $\mathbb{E}\left[Crt_{N}\left(M,E\right)\right]=I+II$, where
	\begin{align*}
	I=&\frac{\omega_{N-2}\sqrt{N}\left(-1\right)^{N-1}\left(p-1\right)^{\frac{N-1}{2}}}{2^{N-\frac{1}{2}}\pi^{\frac{N}{2}}}\int_M dm\left(1-m^2\right)^{\frac{N-3}{2}}\\
	&\int_Edx\exp\left({-\frac{N}{2}\left(\lambda^2 m^{2k-2}\left(1-m^2\right)/p+\left(x+\lambda m^k/k\right)^2\right)}\right)h_{N-1}\left(\sqrt{N}y\right),
	\end{align*}
	and
	\begin{align*}
	II=&\frac{\omega_{N-2}N\left(-1\right)^{N-1}\left(p-1\right)^{\frac{N-1}{2}}}{2^{N-\frac{3}{2}}\pi^{\frac{N}{2}}}\int_{M}dm\left(1-m^2\right)^{\frac{N-3}{2}}\\
	&\int_Edx\exp\left({-\frac{N}{2}\left(\lambda^2 m^{2k-2}\left(1-m^2\right)/p+\left(x+\lambda m^k/k\right)^2\right)}\right)\theta h_{N-2}\left(\sqrt{N}y\right).
	\end{align*}
	We consider term $I$ first.  Using the Hermite function $\phi_{N-1}$ (see Definition \ref{def_H}), 
\begin{align}
&I=\frac{\omega_{N-2}N^{\frac{1}{2}}\left(-1\right)^{N-1}\left(2^{N-1}\left(N-1\right)!\sqrt{\pi}\right)^{\frac{1}{2}}}{2^{N-\frac{1}{2}}\pi^{\frac{N}{2}}\sqrt{p-1}}\int_Mdm\left(1-m^2\right)^{-\frac{3}{2}}\\ &\int_Edx\exp\left({\frac{N}{2}\left(y^2+\log\left(1-m^2\right)\left(p-1\right)-\lambda^2m^{2k-2}\left(1-m^2\right)/p-\left(x+\lambda m^k/k\right)^2\right)}\right)\phi_{N-1}\left(\sqrt{N}y\right).\nonumber
\end{align}
	Since for any $m\in M$ and $x\in E$,  $y=y\left(x,m\right)<-\sqrt{2}$. Using Lemma \ref{lem_PR} and let $\tilde{h}(y)=\frac{\sqrt{2}h(y)}{\sqrt{y^2-2}-y}$, we have as $N\to \infty$,
	\begin{align*}
	&I\sim\frac{\omega_{N-2}N^{\frac{1}{2}}\left(2^{N-1}\left(N-1\right)!\sqrt{\pi}\right)^{\frac{1}{2}}}{\sqrt{4\pi\sqrt{2N}}2^{N-\frac{1}{2}}\pi^{\frac{N}{2}}\sqrt{p-1}}\int_{M}dm\left(1-m^2\right)^{-\frac{3}{2}} \\&\int_Edx\exp\left(\frac{N}{2}\left(y^2+\log\left(1-m^2\right)\left(p-1\right)-\lambda^2m^{2k-2}\left(1-m^2\right)/p-\left(x+\lambda m^k/k\right)^2-2I_1\left(-y\right)\right)\right)\tilde{h}(y)\\
	&\sim\frac{\left(N-1\right)\pi^{\frac{N-1}{2}}N^{\frac{1}{4}}\sqrt{\left(N-1\right)!}}{\Gamma\left(\frac{N+1}{2}\right)2^{\frac{N}{2}+\frac{5}{4}}\pi^{\frac{N}{2}+\frac{1}{4}}\sqrt{p-1}}\int_M\int_E\left(1-m^2\right)^{-\frac{3}{2}}\tilde{h}(y\left(x,m\right))\exp\left({NS_{p,k}\left(x,m\right)}\right)dxdm \\
	&\sim\frac{N}{2\sqrt{2}\pi\sqrt{p-1}}\int_M\int_E\left(1-m^2\right)^{-\frac{3}{2}}\tilde{h}(y\left(x,m\right))\exp\left({NS_{p,k}\left(x,m\right)}\right)dxdm.
	\end{align*}
Substituting $y$ for $x$, we have as $N\to \infty$,
\begin{align}
	I\sim \frac{N}{2\pi\sqrt{p}}\int_{M}\int_{\tilde{E}_m}\left(1-m^2\right)^{-\frac{3}{2}}\tilde{h}\left(y\right)\exp\left(N\tilde{S}_{p,k}\left(m,y\right)\right)dydm, \label{1st term}
\end{align}
where $\tilde{E}_m:=\sqrt{\frac{p}{2\left(p-1\right)}}\left(E-\lambda m^k\left(\frac{1}{p}-\frac{1}{k}\right)\right).$

	Similarly, 
	\begin{align*}
&II=\frac{2\lambda\omega_{N-2}N\sqrt{N-1}\left(-1\right)^{N-1}\left(p-1\right)^{\frac{N-2}{2}}\left(k-1\right)}{2^{N}\pi^{\frac{N}{2}}\sqrt{p}\sqrt{N-1}}\\&\int_M\int_E\left(1-m^2\right)^{\frac{N-1}{2}}m^{k-2}\exp\left(-\frac{N}{2}\left(\frac{\lambda^2}{p}m^{2k-2}\left(1-m^2\right)+\left(x+\frac{\lambda m^k}{k}\right)^2\right)\right)h_{N-2}\left(\sqrt{N}y\right)dxdm\\
&= \frac{\lambda \omega_{N-2}N\sqrt{\left(N-2\right)!}\pi^{\frac{1}{4}}\left(k-1\right)\left(-1\right)^{N-1}}{\sqrt{p\left(2\pi\right)^N}}\int_Mdmm^{k-2}\left(p-1\right)^{\frac{N-2}{2}}\left(1-m^2\right)^{\frac{N-1}{2}}\\
&\int_Edx\exp\left(-\frac{N}{2}\left(\frac{\lambda^2}{p}m^{2k-2}\left(1-m^2\right)+\left(x+\frac{\lambda m^k}{k}\right)^2-y^2\right)\right)\phi_{N-2}\left(\sqrt{N}y\right).
\end{align*}
When $N$ is large enough, $\sqrt{\frac{N}{N-1}}y<-\sqrt{2}$, so by Lemma \ref{lem_PR}, 
\[\phi_{N-2}\left(\sqrt{N}y\right)\sim \left(-1\right)^{N-2}\frac{e^{-\left(N-1\right)I_1\left(-\sqrt{\frac{N}{N-1}}y\right)}}{\sqrt{4\pi\sqrt{2\left(N-1\right)}}}\tilde{h}\left(\sqrt{\frac{N}{N-1}}y\right).\]
Therefore as $N\to \infty$,
\begin{align*}
II\sim-\frac{\lambda N\left(k-1\right)}{2\sqrt{2p}\pi}\int_Mdm\left(p-1\right)^{\frac{N-2}{2}}\left(1-m^2\right)^{\frac{N-1}{2}}m^{k-2}\int_{E}dx\tilde{h}\left(\sqrt{\frac{N}{N-1}}y\left(x\right)\right)L_N\left(m,x\right),
\end{align*}
where 
\begin{align*}
L_N\left(m,x\right)=\exp\left(-\frac{N}{2}\left(\frac{\lambda^2}{p}m^{2k-2}\left(1-m^2\right)+\left(x+\lambda m^k/k\right)^2-y^2+\frac{2\left(N-1\right)}{N}I_1\left(-\sqrt{\frac{N}{N-1}}y\right)\right)\right).
\end{align*}
Let $z=\sqrt{\frac{N}{N-1}}y\left(x\right)$, then
\begin{align*}
&\tilde{L}_N\left(m,z\right):=\left(\left(p-1\right)\left(1-m^2\right)\right)^{\frac{N-1}{2}}L_N\left(m,x\right)=\exp\left(-\left(N-1\right)\tilde{S}_{p,k}\left(m,z\right)\right)J_N\left(m,z\right),
\end{align*}
where
\begin{align*}
J_N\left(m,z\right)=\exp\left(-\frac{\lambda^2}{2p^2}m^{2k-2}\left(p\left(1-m^2\right)+m^2\right)+\frac{\lambda m^kz}{p}\sqrt{\frac{2\left(p-1\right)}{p}}\left(N-1\right)\left(1-\sqrt{\frac{N}{N-1}}\right)\right).
\end{align*}
Substituting $z$ for $x$, we have as $N\to \infty$,
\begin{align*}
II\sim -\frac{\lambda \left(N-1\right)\left(k-1\right)}{2p\pi}\int_Mdm m^{k-2}\int_{\tilde{E}_{m,N}}dz\tilde{h}\left(z\right)\exp\left(\left(N-1\right)\tilde{S}_{p,k}\left(m,z\right)\right)J_N\left(m,z\right).
\end{align*} 
where $\tilde{E}_{m,N}=\sqrt{\frac{N}{N-1}}\tilde{E}_m$.

Since $\tilde{E}_m$ is precompact, $1_{\tilde{E}_{m,N}}\left(z\right)J_N\left(m,z\right)$ converges to $1_{\tilde{E}_m}\left(z\right)J\left(m,z\right)$ uniformly on $m\in M$ and $z\in \tilde{E}_m$. Therefore as $N\to \infty$,
\begin{align}
II&\sim -\frac{\lambda \left(N-1\right)\left(k-1\right)}{2p\pi}\int_Mdm m^{k-2}\int_{\tilde{E}_{m}}dy\tilde{h}\left(y\right)\exp\left(\left(N-1\right)\tilde{S}_{p,k}\left(m,y\right)\right)J\left(m,y\right).
\label{Eq. II}
\end{align}	
Combining Eq. $\left(\ref{1st term}\right)$ and $\left(\ref{Eq. II}\right)$, we get Eq. $\left(\ref{eq:sharp}\right)$ from the Laplace method.
\end{proof}

We end the section with the proof of Proposition \ref{prop_dom}.

\begin{proof}[Proof of Proposition \ref{prop_dom}]
	
	Applying Kac-Rice formula (Theorem 12.1.1 in \cite{AT2007}) to the Hamiltonian $\left(\ref{res H}\right)$, we have
	\begin{equation}
	\mathbb{E}\left[Crt_{N,\ell}\left(M,E\right)\right]=\int_{\sigma\cdot v_0\in M}\mathbb{E}\left[ \left|\det \nabla^2 f\left(\sigma\right) \right|\cdot 1_{\left\{f\in \sqrt{N}E, i \left(\nabla^2 f\right)=l\right\}}\mid \nabla f=0 \right]\phi_{\nabla f\left(\sigma\right)}\left(0\right)d\sigma, 
	\end{equation}
Set 
\[
A_{N,\ell}(\sigma)= \mathbb{E}\left[ \left|\det \nabla^2 f\left(\sigma\right) \right|\cdot 1_{\left\{f\in \sqrt{N}E, i \left(\nabla^2 f\right)=l\right\}}\mid \nabla f=0 \right].
\]
We now show that for any $\sigma$ with $\sigma \cdot v_{0} \in M \subseteq (m_{\lambda},1)$ and $E$ satisfying \eqref{cond:E}, we have for $\ell \geq 1$

\[
\frac{1}{N} \log A_{N,\ell}(\sigma)= o\left(\frac{1}{N} \log A_{N,0}(\sigma)\right)
\]
uniformly in $\sigma$. Looking at \eqref{eq:decomp1}, \eqref{eq:decomp2}, and \eqref{y-x}, and using Lemma \ref{lem_1st dist} it suffices to show 
there exists $\eta>0$, independent of $y \in E$, such that 
\begin{equation}\label{eq:goalofthisbound}
\frac{\mathbb E \left[ |\det (M -\theta e_{N-1}e_{N-1}^{T} - y I_{N})| \mathbf 1 \{\lambda_{\ell} \leq y \} \right] }{\mathbb E \left[ |\det (M -\theta e_{N-1}e_{N-1}^{T} - y I_{N})| \mathbf 1 \{\lambda_{0} \leq y \} \right] }  \leq \exp (-N \eta).
\end{equation}

Let $L_{N}$ be the empirical spectral measure of the matrix $M-\theta e_{N-1}e_{N-1}^{T}$, $\lambda_{\ell}(\theta)$ its $\ell$-th smallest eigenvalue, and $\mu$ denote the semi-circle law. 
For $\delta>0$ consider the event 
\[
B_{N}(\delta)= \left\{ \left| \int \log |x-y| dL_{N}(x) - \int  \log |x-y| d\mu(x)\right| > \delta \right\}.
\]
By \cite{GBA}, and an application of eigenvalue interlacement, there exist $\epsilon>0$ so that for all $N$ sufficiently large
\[
 \mathbb P \left( B_{N}(\delta) \right) \leq e^{-\epsilon N^{2}}.
\]
Now writing 
\[
|\det (M -\theta e_{N-1}e_{N-1}^{T} - y I_{N})| = \int \log |x-y| dL_{N}(x),
\]
note that there exists $C>0$ so that $ \mathbb E\int \log |x-y| dL_{N}(x) \leq \exp (CN)$  and  a positive constant $C'$, such that for $N$ large enough 
\begin{align}\label{estimate1}
\mathbb E \left[ |\det (M -\theta e_{N-1}e_{N-1}^{T} - y I_{N})| \mathbf 1 \{\lambda_{\ell} \leq y \} \right] 
&= \mathbb E \left[ |\det (M -\theta e_{N-1}e_{N-1}^{T} - y I_{N})| \mathbf 1 \{\lambda_{\ell} \leq y \}  \mathbf 1\{ B_{N}(\delta) \}\right] \nonumber \\
&+\mathbb E \left[ |\det (M -\theta e_{N-1}e_{N-1}^{T} - y I_{N})| \mathbf 1 \{\lambda_{\ell} \leq y \}  \mathbf 1\{ B_{N}^{c}(\delta) \}\right] \nonumber \\
&\leq e^{N\left( \delta + \int \log |x-y| d\mu\right)}\mathbb P(\lambda_{\ell} \leq y) + e^{-\epsilon N^{2} +C'N}.
\end{align}
At the same time, we also have the lower bound
\begin{align}\label{estimate2}
\mathbb E \left[ |\det (M -\theta e_{N-1}e_{N-1}^{T} - y I_{N})| \mathbf 1 \{\lambda_{0} \leq y \} \right] 
&\geq  e^{N\left( -\delta + \int \log |x-y| d\mu\right)}\mathbb P(\lambda_{0}\leq y)(1-e^{-\epsilon N^{2}}).
\end{align}
Thus, for $N$ large enough, we obtain for all $y \in E$
\begin{align}
\frac{\mathbb E \left[ |\det (M -\theta e_{N-1}e_{N-1}^{T} - y I_{N})| \mathbf 1 \{\lambda_{\ell} \leq y \} \right] }{\mathbb E \left[ |\det (M -\theta e_{N-1}e_{N-1}^{T} - y I_{N})| \mathbf 1 \{\lambda_{0} \leq y \} \right] } \leq \exp\left(N2\delta - \frac{\epsilon}{2} N^{2}\right) \frac{ \mathbb P(\lambda_{\ell}\leq y)}{ \mathbb P(\lambda_{0}\leq y)}. \label{eq:Sua2}
\end{align}
On the other hand, by our choice of $E$, there exists $\kappa>0$ such that $y < -\sqrt{2}+\kappa$ for all $y \in E$.
 By an application of the large deviation principle for the extreme eigenvalues of rank one perturbation of GOE \cite[Theorem 2.13]{Alice}, there exists $\rho>0$ depending only on $\kappa$ so that 
\begin{align}\label{eq:Sua4}
 \frac{ \mathbb P(\lambda_{\ell}\leq y)}{ \mathbb P(\lambda_{0}\leq y)} \leq \exp(-N \rho).
\end{align}
Plugging \eqref{eq:Sua4} into \eqref{eq:Sua2}, we find that there exists $\eta>0$ so that for $N$ large enough, and all $y\in E$ the bound \eqref{eq:goalofthisbound} is satisfied. This completes the proof of the proposition. 
\end{proof}

\section{The mean number of deep minima}\label{sec.dm}

In this section we prove Theorems \ref{thm1} and \ref{Prop_lamda}. 

\begin{proof}[Proof of Theorem \ref{thm1}] If $x_*\left(\lambda\right)\in E \text{ and }m_*\left(\lambda\right)\in M$, we prove Theorem \ref{thm1} by deriving an explicit formula for the constant $C\left(\lambda,p,k\right)$ as follows, 
\begin{align}
&\lim_{N\to \infty}\mathbb{E}[Crt_N\left(M,E\right)]=C\left(\lambda,p,k\right)=\frac{\sqrt{2}\left(\sqrt{p}\left(1-m_*\right)^{-\frac{3}{2}}-\lambda\left(k-1\right)m_*^{k-2}\right)h\left(y_*\right)}{p\left(\sqrt{y_*^2-2}-y_*\right)\sqrt{\left|\partial_{yy}\tilde{S}_{p,k}\left(m_*,y_*\right)g^{\prime\prime}\left(m_*\right)\right|}},\label{eq-contl}
\end{align}
where
\begin{equation}
y_*:=y_*\left(m_*\right),  y_*\left(m\right)=\frac{\lambda m^k}{2\sqrt{p}}\frac{p-2}{\sqrt{2\left(p-1\right)}}-\frac{p}{\sqrt{2\left(p-1\right)}}\sqrt{\left(\frac{\lambda m^k}{2\sqrt{p}}\right)^2+1},
\end{equation} $g\left(m\right)=\tilde{S}_{p,k}\left(m,y_*\left(m\right)\right)$, $h\left(\cdot\right)$ and  $I_1\left(\cdot\right)$ are defined in Theorem \ref{THM PR}.

Otherwise we show that $\tilde{S}_{p,k}\left(m_o,y_o\right)<0$ in Eq.(\ref{eq:sharp}). Therefore, 
\begin{align*}
	\lim_{N\to \infty}\mathbb{E}[Crt_N\left(M,E\right)]=0.
\end{align*}
A direct computation gives
\begin{equation}\label{eq:partialy}
-\partial_y\tilde{S}_{p,k}=\frac{p-2}{p}y+\frac{\lambda m^k}{p}\sqrt{\frac{2\left(p-1\right)}{p}}-\sqrt{y^2-2}.
\end{equation}
Let $A=\frac{p-2}{p}$, $B=\frac{\lambda m^k}{p}\sqrt{\frac{2\left(p-1\right)}{p}}$, then $-\partial_y\tilde{S}_{p,k}=Ay+B-\sqrt{y^2-2}$.

When $m\leq m_{\lambda}$, $-\frac{B}{A}\geq -\sqrt{2}$. Therefore $\partial_y\tilde{S}_{p,k}\geq 0$ on $\left(-\infty,-\sqrt{2}\right)$ and $\tilde{S}_{p,k}\left(m,\cdot\right)$ is increasing. 

When  $m\geq m_{\lambda}$, $-\frac{B}{A}\leq -\sqrt{2}$. Therefore   $\tilde{S}_{p,k}\left(m,\cdot\right)$ has a unique maximum in $\left(-\infty,-\sqrt{2}\right)$ at 
\begin{equation}
y_*\left(m\right)=\frac{AB-\sqrt{2+B^2-2A^2}}{1-A^2}=\frac{\lambda m^k}{2\sqrt{p}}\frac{p-2}{\sqrt{2\left(p-1\right)}}-\frac{p}{\sqrt{2\left(p-1\right)}}\sqrt{\left(\frac{\lambda m^k}{2\sqrt{p}}\right)^2+1}.\label{GSC1}
\end{equation}
We then define $g\left(m\right)=\tilde{S}_{p,k}\left(m,y_*\left(m\right)\right)$.
Plugging Eq. $\left(\ref{GSC1}\right)$ into $\tilde{S}_{p,k}\left(m,\cdot\right)$, we have
\begin{equation*}
g\left(m\right)=l\left(v\right)=\frac{1}{2}\log\left(1-m^2\right)+\left(1-2/m^2\right)v^2+v\sqrt{v^2+1}-\log\left(v+\sqrt{v^2+1}\right),
\end{equation*}
where $v=\frac{\lambda m^k}{2\sqrt{p}}$.

We compute $l^\prime\left(v\right)=2v\left(1-\frac{2}{m^2}+\frac{\sqrt{v^2+1}}{v}\right)$. Since $\frac{\sqrt{v^2+1}}{v}$ is decreasing, $l^\prime\left(v\right)=0$ on $\left(0,\infty\right)$ if and only if $v=\frac{m^2}{2\sqrt{1-m^2}}$  and 
\[l_{\max}=l\left(\frac{m^2}{2\sqrt{1-m^2}}\right)=0.\]

The maximum is achieved if and only if Eq.$\left(\ref{eq:minima}\right)$ holds, i.e. $\frac{\lambda m^k}{\sqrt{p}}=\frac{m^2}{\sqrt{1-m^2}}$.

As  mentioned in Lemma \ref{lem.UniqueSol}, when  $\lambda\geq \max\left\{\lambda^{\left(1\right)}\left(p,k\right), \lambda^{\left(2\right)}\left(p,k\right)\right\}$, there is a unique solution $m_*$ of Eq. $\left(\ref{eq:minima}\right)$ such that $m_*\geq m_\lambda$, 
and the above computation implies that
\begin{align}
\sup_{m_{\lambda}\leq  m < 1,y<-\sqrt{2}}\tilde{S}_{p,k}\left(m,y\right)=\sup_{m_{\lambda}\leq  m < 1}\tilde{S}_{p,k}\left(m,y_*\left(m\right)\right)=\tilde{S}_{p,k}\left(m_*,y_*\left(m_*\right)\right)=0.\label{high-l-max}
\end{align}
Recall that $\tilde{E}_m\subset \left(-\infty,-\sqrt{2}\right)$  and $y_*\left(m\right)\in \tilde{E}_m$  for any $m$. By Laplace's method, as $N\to \infty$, 
\begin{align}
I&\sim \frac{N}{2\pi\sqrt{p}}\int_{M}\sqrt{\frac{2\pi}{N\left|\partial_{yy}\tilde{S}_{p,k}\left(m,y_*\left(m\right)\right)\right|}}\left(1-m^2\right)^{-\frac{3}{2}}\frac{\sqrt{2}h\left(y_*\left(m\right)\right)}{\sqrt{y_*\left(m\right)-2}-y_*\left(m\right)}\exp\left(Ng\left(m\right)\right)dm\nonumber\\
& \sim \frac{1}{\sqrt{p\left|\partial_{yy}\tilde{S}_{p,k}\left(m_*,y_*\right)g^{\prime\prime}\left(m_*\right)\right|}} \left(1-m_*^2\right)^{-\frac{3}{2}}\frac{\sqrt{2}h\left(y_*\right)}{\sqrt{y_*-2}-y_*}\exp\left(Ng\left(m_*\right)\right)\nonumber\\
&=\frac{\sqrt{2}h\left(y_*\right)\left(1-m_*^2\right)^{-\frac{3}{2}}}{\left(\sqrt{y_*-2}-y_*\right)\sqrt{p\left|\partial_{yy}\tilde{S}_{p,k}\left(m_*,y_*\right)g^{\prime\prime}\left(m_*\right)\right|}}.\label{Eq-I}
\end{align}
Similarly, we apply Laplace method to $II$ and get
\begin{align}
II&\sim -\frac{\lambda \left(N-1\right)\left(k-1\right)}{2p\pi}\int_{M}dm\times  \nonumber\\
&\sqrt{\frac{2\pi}{\left(N-1\right)\left|\partial_{yy}\tilde{S}_{p,k}\left(m,y_*\left(m\right)\right)\right|}}m^{k-2}\frac{\sqrt{2}h\left(y_*\left(m\right)\right)}{\sqrt{y_*\left(m\right)-2}-y_*\left(m\right)}\exp\left(\left(N-1\right)g\left(m\right)\right)J\left(m,y_*\left(m\right)\right)\nonumber\\
&\sim \frac{\lambda\left(k-1\right)}{p}\frac{m_*^{k-2}\sqrt{2}h\left(y_*\right)\exp\left(\left(N-1\right)g\left(m_*\right)\right)J\left(m_*,y_*\right)}{\left(\sqrt{y_*-2}-y_*\right)\sqrt{\left|\partial_{yy}\tilde{S}_{p,k}\left(m_*,y_*\right)g^{\prime\prime}\left(m_*\right)\right|}}\nonumber\\
&=\frac{\sqrt{2}\lambda\left(k-1\right)m_*^{k-2}h\left(y_*\right)J\left(m_*,y_*\right)}{p\left(\sqrt{y_*-2}-y_*\right)\sqrt{\left|\partial_{yy}\tilde{S}_{p,k}\left(m_*,y_*\right)g^{\prime\prime}\left(m_*\right)\right|}}.\label{Eq-II}
\end{align}	
 It remains to show that $J\left(m_*,y_*\right)=1$. This is obtained as follows. Using Eq.$\left(\ref{eq:minima}\right)$, 
\begin{align*}
&-\frac{\lambda^2}{2p^2}m_*^{2k-2}\left(p\left(1-m_*^2\right)+m_*^2\right)+\frac{\lambda m_*^ky_*}{2p}\sqrt{\frac{2\left(p-1\right)}{p}}\\
&=\frac{\lambda m_*^k}{2p}\left(p\left(1-m_*^2\right)+m_*^2\right)\left(\frac{1}{\sqrt{p\left(1-m_*^2\right)}}-\frac{1}{\sqrt{p\left(1-m_*^2\right)}}\right)=0.
\end{align*}
Combining this with Eq. $\left(\ref{Eq-I}\right)$ and $\left(\ref{Eq-II}\right)$, we get Eq. $\left(\ref{eq-contl}\right)$.
\end{proof}

\begin{proof}[Proof of Theorem \ref{Prop_lamda}]
	Since $m_*$ satisfies Eq.$\left(\ref{eq:minima}\right)$, as $\lambda\to \infty$, $m_*\to 1$. When $k=1$ or $2$ this can be obtained directly from Eq.$\left(\ref{eq-m*12}\right)$. When $k\geq 3$, \color{black}by Implicit differentiation theorem, denoting $m_*^\prime:=\frac{d}{d\lambda}m_*\left(\lambda\right)$, we have
	\begin{align*}
	&\hspace{1cm}\frac{m_*^k}{\sqrt{p}}+\frac{\lambda k m_*^{k-1}}{\sqrt{p}}m_*^\prime=\left(2m_*\left(1-m_*^2\right)^{-\frac{1}{2}}+m_*^3\left(1-m_*^2\right)^{-\frac{3}{2}}\right)m_*^\prime\\
	&\implies \frac{m_*^k}{\sqrt{p}}= m_*^\prime\left(1-m_*^2\right)^{-\frac{3}{2}}m_*\left(\left(k-1\right)m_*^2-\left(k-2\right)\right).
	\end{align*}
	Since $m_*>\sqrt{\frac{k-2}{k-1}}$, so  $\left(k-1\right)m_*^2-\left(k-2\right)>0$ and thus $m^\prime_*>0$. Therefore,  as $\lambda\to \infty$, $m_*\left(\lambda\right)\uparrow 1$.

	We also have
	\begin{align}
	y_*\left(m_*\right)&=\frac{p-2}{2\sqrt{2\left(p-1\right)}}m_*^2\left(1-m_*^2\right)^{-\frac{1}{2}}-\frac{p}{\sqrt{2\left(p-1\right)}}\sqrt{\frac{m_*^4}{4\left(1-m_*^2\right)}+1}\nonumber\\
	&=\frac{1}{\sqrt{2\left(p-1\right)}}\left(1-m_*^2\right)^{-\frac{1}{2}}\left(\frac{p-2}{2}m_*^2-\frac{p}{2}\left(2-m_*^2\right)\right)\nonumber\\
	&=\frac{1}{\sqrt{2\left(p-1\right)}}\left(1-m_*^2\right)^{-\frac{1}{2}}\left(\left(p-1\right)m_*^2-p\right). \label{y*}
	\end{align}
	Therefore, as $\lambda\to \infty$,
	\begin{align}
	y_*\left(m_*\right)\sim -\frac{1}{\sqrt{2\left(p-1\right)}}\left(1-m_*^2\right)^{-\frac{1}{2}}.\label{eq-y}
	\end{align}
	Note that 
	\begin{align*}
	{y_*}^\prime\left(m\right)=\frac{p-2}{2\sqrt{2p\left(p-1\right)}}\lambda km^{k-1}-\frac{p}{\sqrt{2\left(p-1\right)}}\left(\left(\frac{\lambda m^k}{2\sqrt{p}}\right)^2+1\right)^{-\frac{1}{2}}\frac{\lambda m^k}{2\sqrt{p}}\cdot \frac{\lambda km^{k-1}}{2\sqrt{p}}.
	\end{align*}
	Using Eq.$\left(\ref{eq:minima}\right)$, we have as $\lambda\to \infty$,
	\begin{align}
	y^\prime_*\left(m_*\right)&=\frac{k\left(p-2\right)}{2\sqrt{2\left(p-1\right)}}m_*\left(1-m_*^2\right)^{-\frac{1}{2}}-\frac{kp}{2\sqrt{2\left(p-1\right)}}\frac{m_*^3}{2-m_*^2}\left(1-m_*^2\right)^{-\frac{1}{2}}\nonumber\\
	&=\frac{k}{\sqrt{2\left(p-1\right)}}\left(1-m_*^2\right)^{-\frac{1}{2}}\frac{\left(p-2\right)m_*-\left(p-1\right)m_*^3}{2-m_*^2}\nonumber\\
	&\sim -\frac{k}{\sqrt{2\left(p-1\right)}}\left(1-m_*^2\right)^{-\frac{1}{2}},\label{eq-yp}
	\end{align}
and 
	\begin{align}
	\partial_{yy}\tilde{S}_{p,k}\left(m_*,y_*\right)&=\frac{2-p}{p}+y_*\left(y_*^2-2\right)^{-\frac{1}{2}}\nonumber\\
	&\sim  -\frac{2\left(p-1\right)}{p}.
	\end{align}
	We also compute for $k\geq 1$,
	\begin{align*}
	\partial_m\tilde{S}_{p,k}=-\frac{m}{1-m^2}-\frac{\lambda k m^{k-1}}{p}\sqrt{\frac{2\left(p-1\right)}{p}}y-\frac{\lambda^2\left(k-1\right)}{p}m^{2k-3}+\frac{\lambda^2 k\left(p-1\right)}{p^2}m^{2k-1}.
	\end{align*}
For $k\geq 2$,
	\begin{align*}
	\partial_{mm}\tilde{S}_{p,k}=&-\frac{1+m^2}{\left(1-m^2\right)^2}-\frac{\lambda k\left(k-1\right) m^{k-2}}{p}\sqrt{\frac{2\left(p-1\right)}{p}}y-\frac{\lambda^2\left(k-1\right)\left(2k-3\right)}{p}m^{2k-4}\\
	&+\frac{\lambda^2 k\left(p-1\right)\left(2k-1\right)}{p^2}m^{2k-2},
	\end{align*}
	$\partial_{mm}\tilde{S}_{p,1}=\frac{1+m^2}{-\left(1-m^2\right)^2}+\frac{\left(p-1\right)\lambda^2}{p^2}$, and for $k\geq 2$,
	\begin{align*}
	\partial_{my}\tilde{S}_{p,k}=-\frac{\lambda k m^{k-1}}{p}\sqrt{\frac{2\left(p-1\right)}{p}}.
	\end{align*}
		Using Eq.$\left(\ref{eq:minima}\right)$ and $\left(\ref{eq-y}\right)$, we have as $\lambda\to \infty$, 
			\begin{align}
		\partial_{mm}\tilde{S}_{p,k}\left(m_*,y_*\left(m_*\right)\right)=&-\left(1+m_*^2\right)\left(1-m_*^2\right)^{-2}-
		\frac{k\left(k-1\right)}{p}\left(1-m_*^2\right)^{-1}\left(\left(p-1\right)m_*^2-p\right)\nonumber\\
		&-\left(k-1\right)\left(2k-3\right)\left(1-m_*^2\right)^{-1}-\frac{k\left(p-1\right)\left(2k-1\right)m_*^2}{p}\left(1-m_*^2\right)^{-1}\nonumber\\
		&\sim -2\left(1-m_*^2\right)^{-2},\label{eq-mm}
		\end{align}
		and
		\begin{align}
		\partial_{my}\tilde{S}_{p,k}\left(m_*,y_*\left(m_*\right)\right)=\frac{k\sqrt{2\left(p-1\right)}}{p}m_*\left(1-m^2_*\right)^{-\frac{1}{2}}\sim \frac{k\sqrt{2\left(p-1\right)}}{p}\left(1-m^2_*\right)^{-\frac{1}{2}}.\label{eq-my}
		\end{align}
		Recall that $g\left(m\right)=\tilde{S}_{p,k}\left(m,y_*\left(m\right)\right)$, so 
		\begin{align}
		g^{\prime\prime}=\partial_{mm}\tilde{S}_{p,k}+2\partial_{my}\tilde{S}_{p,k}\cdot y_*^\prime+\partial_{yy}\left(y_*^\prime\right)^2+\partial_y\tilde{S}_{p,k}\cdot y^{\prime\prime}.
		\end{align} 
		Note that $\partial_y\tilde{S}_{p,k}\left(m_*,y_*\left(m_*\right)\right)=0$, using Eq. $\left(\ref{eq-mm}\right)$, $\left(\ref{eq-my}\right)$, $\left(\ref{eq-y}\right)$ and $\left(\ref{eq-yp}\right)$, we know that as $\lambda\to \infty$,
		\begin{align}
		g^{\prime\prime}\left(m_*\right)&\sim -2\left(1-m_*^2\right)^{-2}+\frac{2k\sqrt{2\left(p-1\right)}}{p}\left(1-m_*^2\right)^{-\frac{1}{2}}\cdot\left(\frac{-k}{\sqrt{2\left(p-1\right)}}\right)\left(1-m_*^2\right)^{-\frac{1}{2}}\nonumber\\
		&+\frac{2\left(p-1\right)}{p}\frac{k^2}{2\left(p-1\right)}\left(1-m_*^2\right)^{-1}\nonumber\\
		&\sim -2\left(1-m_*^2\right)^{-2}.\label{eq-g''}
		\end{align}
	   From the definition of $h\left(\cdot\right)$ in Theorem \ref{THM PR} and Eq. $\left(\ref{eq-y} \right)$, it is easy to see
	   \[\lim_{\lambda\to \infty}h\left(y_*\left(m_*\right)\right)=2.\]
	   To sum up, as $\lambda\to \infty$,
	   \begin{align}
	   C\left(\lambda,p,k\right)&\sim 
	   2\frac{\left(\sqrt{p}\left(1-m_*^2\right)^{-\frac{3}{2}}-\lambda\left(k-1\right)m_*^{k-2}\right)}{p\sqrt{2}y_*\left(m_*\right)\sqrt{\left|\partial_{yy}\tilde{S}_{p,k}\left(m_*,y_*\right)g^{\prime\prime}\left(m_*\right)\right|}}\nonumber\\
	   &\sim\frac{2\left(\sqrt{p}\left(1-m_*^2\right)^{-\frac{3}{2}}-\left(k-1\right)\sqrt{p}\left(1-m_*^2\right)^{-\frac{1}{2}}\right)}{-p\sqrt{2}y_*\left(m_*\right)\sqrt{2\frac{p-1}{p}\cdot 2\left(1-m_*^2\right)^{-2}}}\nonumber\\
	   &\sim \frac{1}{\sqrt{p-1}}\left(1-m_*^2\right)^{-\frac{1}{2}}\cdot \sqrt{\left(p-1\right)}\left(1-m_*^2\right)^{\frac{1}{2}}\\
	   &=1.
	   \end{align}

\end{proof}
\section{Limiting ground state energy}\label{sec.GSE}
In this section we prove Theorem \ref{THM_GSE}. The proof relies on the following two propositions whose proofs are presented after the proof of Theorem \ref{THM_GSE}. 
\begin{proposition}\label{Prop_UB}
	For any $m\in \left(0,1\right)$, 
	\begin{equation}
	\limsup_{N\to \infty}\mathbb{E}\left[\frac{1}{N}\min_{\sigma \in S^{N-1}\left(\sqrt{N}\right)} H_{N}(\sigma)\right]\leq -\frac{\lambda m^k}{k}-\sqrt{p\left(1-m^2\right)}.
	\end{equation}
\end{proposition}
\begin{proposition}\label{Prop_LB}
		\begin{equation}
		\liminf_{N\to \infty}\frac{1}{N}\min_{\sigma \in S^{N-1}\left(\sqrt{N}\right)} H_{N}(\sigma)\geq \lambda m_*^k\left(\frac{1}{2}-\frac{1}{k}\right)-\sqrt{\frac{\lambda^2m_*^{2k}}{4}+p}  \text{      }\text{       }a.s..\label{LB}
		\end{equation}
\end{proposition}
\begin{proof}[Proof of Theorem \ref{THM_GSE} assuming Proposition \ref{Prop_LB} and \ref{Prop_UB}]
	By Gaussian concentration inequality and Borel-Canteli lemma, 
	\[\lim_{N\to \infty}\frac{1}{N}\min_{\sigma \in S^{N-1}\left(\sqrt{N}\right)} H_{N}(\sigma)=\lim_{N\to \infty}\mathbb{E}\left[\frac{1}{N}\min_{\sigma \in S^{N-1}\left(\sqrt{N}\right)} H_{N}(\sigma)\right] a.s..\]
	Therefore it suffices to show 
	\begin{align}
	-\frac{\lambda m_*^k}{k}-\sqrt{p\left(1-m_*^2\right)}=\lambda m_*^k\left(\frac{1}{2}-\frac{1}{k}\right)-\sqrt{\frac{\lambda^2m_*^{2k}}{4}+p}.\label{eq:lu}
	\end{align}
	Using Eq. $\left(\ref{eq:minima}\right)$, 
	\begin{align*}
	\lambda m_*^k/2-\sqrt{\frac{\lambda^2m_*^{2k}}{4}+p}&=\frac{\sqrt{p}m_*^2}{2\sqrt{1-m_*^2}}-\sqrt{\frac{p\left(m_*^4-4m_*^2+4\right)}{4\left(1-m_*^2\right)}}\\
	&=\frac{\sqrt{p}\left(m_*^2+m_*^2-2\right)}{2\sqrt{1-m_*^2}}=-\sqrt{p\left(1-m_*^2\right)}.
	\end{align*}
Therefore Eq.$\left(\ref{eq:lu}\right)$ holds.
\end{proof}

Now we prove Proposition \ref{Prop_UB} and \ref{Prop_LB}.
\begin{proof}[Proof of Proposition \ref{Prop_UB}]

	For any $m\in \left(0,1\right)$,
	\begin{align*}
	\frac{1}{N}\min_{\sigma \in S^{N-1}\left(\sqrt{N}\right),\sigma\cdot \textbf{v}_0=m} H_{N}(\sigma)\geq \frac{1}{N}\min_{\sigma \in S^{N-1}\left(\sqrt{N}\right)} H_{N}(\sigma),
	\end{align*}
	so 
	\[\mathbb{E}\left[\frac{1}{N}\min_{\sigma \in S^{N-1}\left(\sqrt{N}\right),\sigma\cdot \textbf{v}_0=m} H_{N}(\sigma)\right]\geq \mathbb{E}\left[\frac{1}{N}\min_{\sigma \in S^{N-1}\left(\sqrt{N}\right)} H_{N}(\sigma)\right].\]
	 Since $H_{N}$ is isotropic, without loss of generality, we assume $\textbf{v}_0=\sqrt{N}e_N$, then conditional on $\sigma_N=\sqrt{N}m$, 
	 \[H_{N}\left(\sigma\right)=-\lambda N\frac{m^k}{k}-\sqrt{N}J_{NN\dots N}m^p-\frac{1}{N^{\frac{p-1-l}{2}}}\sum_{l=0}^{p-1}m^l\sum_{i_{k_j}=N,1\leq i_k\leq N-1, k\neq k_j, j\in \left[l\right]}J_{i_1,i_2,\dots,i_p}\frac{\sigma_{i_1}\sigma_{i_2}\cdots \sigma_{i_p}}{\sigma_{i_{k_1}}\sigma_{i_{k_2}}\cdots \sigma_{i_{k_l}}}.\]
	 Since for different sets of $\left(i_{k_j}\right)_{j=1}^l$, $J_{i_1,i_2,\dots,i_p}$ are i.i.d, so 
	 \begin{align*}
	 H_{N}\left(\sigma\right)&\overset{\left(d\right)}{=}-\lambda N\frac{m^k}{k}-\sqrt{N}J_{NN\dots N}m^p\\
	 &-\sqrt{N}\sum_{l=0}^{p-1}{p \choose l}^{\frac{1}{2}}m^l\left(1-m^2\right)^{\frac{p-l}{2}}\sum_{1\leq i_1,i_2,\dots i_{p-l}\leq N-1}g_{i_1,i_2,\dots, i_{p-1}}\hat{\sigma}_{i_1}\hat{\sigma}_{i_2}\cdots \hat{\sigma}_{i_{p-l}}
	 \end{align*}
	 where $\hat{\sigma}_{k}={\sigma}_{k}/\sqrt{N\left(1-m^2\right)}, k\in \left[N-1\right]$.
	 
	 Note that $\sum_{k=1}^{N-1}\hat{\sigma}_{k}^2=1$, therefore
	 \begin{equation*}
	 -\sqrt{N}\sum_{l=0}^{p-1}{p-1 \choose l}^{\frac{1}{2}}m^l\left(1-m^2\right)^{\frac{p-l}{2}}\sum_{1\leq i_1,i_2,\dots i_{p-l}\leq N-1}g_{i_1,i_2,\dots, i_{p-1}}\hat{\sigma}_{i_1}\hat{\sigma}_{i_2}\cdots \hat{\sigma}_{i_{p-l}}
	 \end{equation*}
	 is a spherical mixed p-spin model with mixture \[\xi\left(x\right)=\sum_{l=0}^{p-1}{p \choose l}m^{2l}\left(1-m^2\right)^{p-l}x^{p-l}=\left(m^2+\left(1-m^2\right)x\right)^p-m^{2p}.\]
	 By Proposition 1 in \cite{CS2015}(see also Theorem 1.10 in \cite{Jagannath17}),
	 \begin{align*}
	 \mathbb{E}\left[\frac{1}{N}\min_{\sigma \in S^{N-1}\left(\sqrt{N}\right),\sigma\cdot \textbf{v}_0=m} H_{N}(\sigma)\right]=-\frac{\lambda m^k}{k}-\sqrt{\xi^\prime\left(1\right)}=-\frac{\lambda m^k}{k}-\sqrt{p\left(1-m^2\right)}.
	 \end{align*}
 \end{proof}
Recall that in this paper we reserve the symbol $x_*$ for the the right hand side of Eq.$\left(\ref{LB}\right)$, i.e.
	\[x_*:=\lambda m_*^k\left(\frac{1}{2}-\frac{1}{k}\right)-\sqrt{\frac{\lambda^2m_*^{2k}}{4}+p}.\]
The key to proving Proposition \ref{Prop_LB} is to identify the point at which 0, the supremum of the complexity function $S_{p,k}$,  is attained. The following proposition shows that the point lies in the high-latitude region of the sphere.
\begin{proposition}\label{prop.low-latitude-complexity}There exists a constant $\tilde{\lambda}_c=\tilde{\lambda}_c\left(p,k\right)$ such that for any $\lambda \geq \tilde{\lambda}_c $,  $M=\left(0,m_{\lambda}\right)$, 
	\begin{align*}
	\sup_{m\in \bar{M}}S_{p,k}\left(m,x_*\right)< 0.
	\end{align*} 
\end{proposition}
\begin{proof}
	Using the correspondence between $y_*$ and $x_*$ $\left(\text{see Eq.      }\left(\ref{y}\right)\right)$,  we have
	\begin{align*}
	f\left(m\right):=\tilde{S}_{p,k}\left(m,y_*\right)=S_{p,k}\left(m,x_*\right).
	\end{align*}
	We will first show that $f\left(m\right)$ has at most one critical point on  $M$, and if it exists,  it must be  a local minimum of $f$, then we use the results on the pure p-spin model from \cite{ABC13} and Theorem \ref{thm1} to show that $f\left(0\right)<0$ and $f\left(m_\lambda\right)<0$, thus deriving $\sup_{m\in \bar{M}}S_{p,k}\left(m,x_*\right)=\sup_{m\in \bar{M}}f\left(m\right)<0$.

	A direct computation shows that
	\begin{align*}
	f^\prime\left(m\right)&=-\frac{m}{1-m^2}-\frac{\lambda k m^{k-1} }{p}\sqrt{\frac{2\left(p-1\right)}{p}}y_*-\frac{\lambda^2\left(k-1\right)m^{2k-3}}{p}+\frac{\left(p-1\right)\lambda^2km^{2k-1}}{p^2}\\
	&=-\frac{m}{1-m^2}f_1\left(m\right),
	\end{align*}
	where 
	\begin{align*}
	f_1\left(m\right)=1+\frac{\lambda k m^{k-2}\left(1-m^2\right) }{p}\sqrt{\frac{2\left(p-1\right)}{p}}y_*+\frac{\lambda^2\left(k-1\right)m^{2k-4}\left(1-m^2\right)}{p}-\frac{\left(p-1\right)\lambda^2km^{2k-2}\left(1-m^2\right)}{p^2}.
	\end{align*}
	Let $u:=u\left(\lambda,m\right):=\lambda m^{k-2}\left(1-m^2\right)$.
	
	Case I: If $k\leq p$,
	\begin{align*}
	f_1\left(m\right)&=1+\frac{k}{p}\sqrt{\frac{2\left(p-1\right)}{p}}y_*u+\frac{k-1}{p}u^2+\frac{\lambda^2m^{2k-2}\left(1-m^2\right)}{p^2}(k-p)\leq f_2\left(u\right)
	\end{align*}
	where 
	\begin{align*}
	f_2\left(u\right)=1+\frac{k}{p}\sqrt{\frac{2\left(p-1\right)}{p}}y_*u+\frac{k-1}{p}u^2.
	\end{align*}
	The larger zero of $f_2$ is
	\begin{align*}
	u_*:=\frac{1}{\left(k-1\right)\sqrt{2p}}\left(-k\sqrt{p-1}y_*+\sqrt{k^2\left(p-1\right)y_*^2-2\left(k-1\right)p^2}\right).
	\end{align*}
	Recall that for fixed $\lambda$, $u=\lambda m^{k-2}\left(1-m^2\right)$, which is increasing with respect to $m$ over $\left[0,\sqrt{\frac{k-2}{k}}\right]$, so when $\lambda\geq p\sqrt{\frac{p-2}{p-1}}\left(\frac{k}{k-2}\right)^\frac{k}{2}$,
	\begin{equation*}
	u_{\max}=u\left(m_\lambda\right)=\lambda^{\frac{2}{k}}\left(\frac{p\sqrt{p-2}}{\sqrt{p-1}}\right)^{\frac{k-2}{k}}\left(1-m_\lambda^2\right).
	\end{equation*}
	Therefore, 
	\begin{align}
	u_{\max}\sim \lambda^{\frac{2}{k}}\left(\frac{p\sqrt{p-2}}{\sqrt{p-1}}\right)^{\frac{k-2}{k}}\text{        as  }\lambda\to \infty.\label{u-max}
	\end{align}
	Note that 
	\begin{align*}
	y_*=\sqrt{\frac{p}{2\left(p-1\right)}}\left(x_*-\lambda m_*^k\left(\frac{1}{p}-\frac{1}{k}\right)\right)=-\frac{\lambda m_*^k}{\sqrt{2p\left(p-1\right)}}-p\sqrt{\frac{1-m_*^2}{2\left(p-1\right)}}.
	\end{align*}
	When $\lambda\to \infty$, it is observed from Eq.$\left(\ref{eq:minima}\right)$ that $\lim_{\lambda\to \infty}m_*\left(\lambda\right)=1$,
	so \begin{align*}
	y_*\sim -\frac{\lambda}{\sqrt{2p\left(p-1\right)}}\text{       as  }\lambda\to \infty,
	\end{align*}
	and thus
	\begin{align}
	u_*\sim -\frac{\lambda k}{p\left(k-1\right)}\text{       as  }\lambda\to \infty.\label{ustar}
	\end{align}
	Combining Eq.$\left(\ref{u-max}\right)$ and $\left(\ref{ustar}\right)$, there exists $\tilde{\lambda}_c>0$ such that if $\lambda\geq \tilde{\lambda}_c$, , $u_{\max}<u_*$, so $f_1\left(m\right)$ crosses $m$-axis at most once over $\left[0,m_\lambda\right]$. Note that $f_1\left(0\right)=1>0$ and it is continuous on $\left[0,m_\lambda\right]$, so $f^\prime\left(m\right)<0$ when $m$ is small and it crosses $m$-axis at most once over $\left[0,m_\lambda\right]$. 
	
	Case II: If $k>p$, then when $\lambda\geq 2^k\left(p-2\right)\sqrt{\frac{p}{p-1}}$, $m_\lambda\leq \frac{1}{2}$. Therefore for any $m\leq m_\lambda$,
	\begin{align*}
	\lambda m^{2k-2}\left(1-m^2\right)\leq u^2,
	\end{align*}
	then we have $f_1\left(m\right)\leq f_3\left(u\right)$, where
	\begin{align*}
	f_3\left(u\right)=1+\frac{k}{p}\sqrt{\frac{2\left(p-1\right)}{p}}y_*u+\left(\frac{k\left(p+1\right)-2p}{p^2}\right)u^2.
	\end{align*}
	The same argument in Case I also applies to Case II and we derive the same conclusion that $f^\prime\left(m\right)<0$ when $m$ is small and it crosses $m$-axis at most once over $\left[0,m_\lambda\right]$. 
	
	This implies 
	\begin{align}
	\sup_{m\in \bar{M}}f\left(m\right)=\max\left\{f\left(0\right), f\left(m_{\lambda}\right)\right\}.\label{GS-l-max}
	\end{align}
	Note that 
	$f\left(0\right)=\Phi_p\left(y_*\right)$, where $\Phi_p\left(\cdot\right)$ is the annealed complexity of  the p-spin spherical spin glass model, see Theorem 2.8 in \cite{ABC13}. It is known that $\Phi_p\left(\cdot\right)$ is an increasing function on $\left(-\infty,-2\sqrt{\frac{p-1}{p}}\right)$ and $\lim_{y\to -\infty}\Phi_p\left(y\right)=-\infty$, so when $\lambda$ is large enough so that $y_*$ is smaller than  the limiting ground state energy of the p-spin spherical spin glass model which is the unique zero of  $\Phi_p\left(\cdot\right)$ on $\left(-\infty, -2\sqrt{\frac{p-1}{p}}\right)$,
	\begin{equation}
	f\left(0\right)=\Phi_p\left(y_*\right)<0.\label{GS-1-pure}
	\end{equation} 
	
	As to $f\left(m_\lambda\right)$,  we know from Theorem \ref{thm1}  (more specifically, Eq. $\left(\ref{high-l-max}\right)$) that when $\lambda\geq \tilde{\lambda}$,
	\[f\left(m_\lambda\right)\leq \sup_{m\geq m_{\lambda}}f\left(m\right)\leq 0.\] Combining this with Eq.(\ref{GS-l-max}) and (\ref{GS-1-pure}) we prove this proposition.
	
\end{proof}
\begin{proof}[Proof of Proposition \ref{Prop_LB}]
	For any $\epsilon>0$, let $M=\left[0,1\right]$ and $E=\left(-\infty,x_*-\epsilon\right)$. It is shown in Theorem \ref{thm1} that for fixed $m<1$, $\tilde{S}_{p,k}\left(m, \cdot\right)$ is increasing on $\left(-\infty, -x_*-\epsilon\right)$. Combining this with Proposition \ref{prop.low-latitude-complexity} we see that 
	\begin{align*}
\limsup_{N\to \infty}\frac{1}{N}\log\mathbb{E}\left[Crt_N\left(M,E\right)\right]&= \sup_{m\in \bar{M},x\in \bar{E}}S_{p,k}\left(m,x\right)\\
&\leq \sup_{m\in \bar{M}}\tilde{S}_{p,k}\left(m,y_*\right)\\
&\leq \max\{\sup_{m\in \left[0,m_\lambda\right]}\tilde{S}_{p,k}\left(m,y_*\right), \sup_{m\in \left[m_\lambda, 1\right]}\tilde{S}_{p,k}\left(m,y_*\right)\}\\
&<0.
\end{align*}
	Therefore, by Markov inequality,
	\[P\left(\frac{1}{N}\min_{\sigma \in S^{N-1}\left(\sqrt{N}\right)} H_{N}(\sigma)\leq x_*-\epsilon\right)\leq P\left(Crt_{N}\left(M,E\right)\geq 1\right)\leq \mathbb{E}\left[Crt_N\left(M,E\right)\right],\]
    then Eq.$\left(\ref{LB}\right)$ follows from Borel-Cantali lemma.

\end{proof}

\appendix
\section{Covariance computations and some formulas from Random Matrix Theory}
In this appendix we derive the random matrices appearing in the Kac-Rice computation in Section \ref{sec.asymptotic} and summarize a series of tools that we use in random matrix theory.

 \begin{lemma}\label{lem_1st dist}
 	Let $f: S^{N-1}\to \mathbb{R}$ be defined in Eq.$(\ref{res H})$. Without loss of generality, we set $\sigma=e_N$, $\hat{\textbf{v}}_0=me_N+\sqrt{1-m^2}e_{N-1}$,  then 

\[\mathbb{E}\left[f\left(\sigma\right)\right]=-\lambda\sqrt{N}m^k/k, Var\left(f\left(\sigma\right)\right)=1.\]
\[\mathbb{E}\left[\nabla f\left(\sigma\right)\right]=-\sqrt{N}\lambda m^{k-1}\sqrt{1-m^2}e_{N-1}\]
\[Cov\left(f\left(\sigma\right),\nabla_if\left(\sigma\right)\right)=Cov\left(\nabla^2_{jk}f\left(\sigma\right),\nabla_if\left(\sigma\right)\right)=0\text{ for }i,j,k=1,2,\dots, N-1.\]
\[Cov\left(\nabla^2f,f\right)=-pI_{N-1}\]
\[Cov\left(\nabla f,\nabla f\right)=pI_{N-1}\]
\[Cov\left(\nabla^2_{ij} f,\nabla^2_{kl} f\right)=p\left(p-1\right)\left(\delta_{ik}\delta_{jl}+\delta_{il}\delta_{jk}\right)+p^2\delta_{ij}\delta_{kl} \text{  for }i,j,k,l=1,2,\dots, N-1.\]
Denote by $\mathbb{E}_{A}$ and $Cov_{A}$ the expectation and covariance  conditional on the event $A$, then
\[\mathbb{E}_{\nabla f\left(\sigma\right)=0}\left[f\left(\sigma\right)\right]=\mathbb{E}\left[f\left(\sigma\right)\right],\]
\[\mathbb{E}_{\nabla f\left(\sigma\right)=0}\left[\nabla^2f\left(\sigma\right)\right]=\mathbb{E}\left[\nabla^2f\left(\sigma\right)\right],\]
\[\mathbb{E}\left[\nabla^2f\left(\sigma\right)\right]=-\sqrt{N}\lambda\left(k-1\right)\left(1-m^2\right)m^{k-2}e_{N-1}e^T_{N-1}+\sqrt{N}\lambda m^k I_{N-1}\]
\[\mathbb{E}_{f=\sqrt{N}x}\left[\nabla^2 f\left(\sigma\right)\right]=-\lambda\sqrt{N}\left(k-1\right)m^{k-2}\left(1-m^2\right)e_{N-1}e^T_{N-1}-pI_{N-1}\left(\sqrt{N}x+\frac{\lambda\sqrt{N}m^k}{k}\right)+\lambda\sqrt{N}m^k I_{N-1}\]
\[Cov_{f=\sqrt{N}x}\left(\nabla^2_{ij}f\left(\sigma\right),\nabla^2_{kl}f\left(\sigma\right)\right)=p\left(p-1\right)\left(\delta_{ik}\delta_{jl}+\delta_{il}\delta_{jk}\right)\text{  for }i,j,k,l=1,2,\dots, N-1.\]	
\end{lemma}
From Lemma \ref{lem_1st dist}, conditional on $\nabla f\left(\sigma\right)=0$, $f\left(\sigma\right)=\sqrt{N}x$, 
\begin{align}
\nabla^2f\left(\sigma\right)\overset{d}{=}&\sqrt{2\left(N-1\right)p\left(p-1\right)}W_{N-1}-\lambda\sqrt{N}\left(k-1\right)m^{k-2}\left(1-m^2\right)e_{N-1}e^T_{N-1}\nonumber\\
&+\sqrt{N}I_{N-1}\left(-px+\left(1-\frac{p}{k}\right)\lambda m^k\right).
\end{align}

\begin{definition}\label{def_H}	
	For $N\in \mathbb{N}$, denote
	\begin{itemize}
		\item Hermite polynomials $h_N\left(x\right)=e^{x^2}\left(-\frac{d}{dx}\right)^Ne^{-x^2}.$
		\item Hermite functions $\phi_N\left(x\right)=\left(2^NN!\sqrt{\pi}\right)^{-\frac{1}{2}}h_N\left(x\right)e^{-\frac{x^2}{2}}$.
	\end{itemize}
\end{definition}
\begin{lemma}[Lemma 3 in \cite{AB13}, Corollary 11.6.3 in \cite{AT2007}]\label{lem_det hermite}
	\[\mathbb{E}\left[\det\left(W_{N-1}-xI_{N-1}\right)\right]=2^{1-N}\left(N-1\right)^{\frac{1-N}{2}}\left(-1\right)^{N-1}h_{N-1}\left(\sqrt{N-1}x\right).\]
	
\end{lemma}
Using Eq.$\left(1.8\right)$ in \cite{DesLiu15}, we obtain the following proposition which is useful for expressing determinants in terms of Hermite polynomials.
\begin{lemma}\label{lem char}
\begin{align}
&\mathbb{E}\left[\det\left(W_{N-1}-fe_{N-1}e^T_{N-1}+sI_{N-1}\right)\right]\nonumber\\
&=\left(\frac{-i}{\sqrt{N-1}}\right)^{N-1}\pi^{-\frac{1}{2}}e^{\left(N-1\right)s^2}\int_{\mathbb{R}}e^{-y^2}\left(y^{N-1}-i\sqrt{N-1}fy^{N-2}\right)e^{2\sqrt{N-1}iys}dy.
\end{align}
\end{lemma}
\begin{remark}
	Setting $f=0$, one can easily recover Lemma \ref{lem_det hermite} using Fourier transform.
\end{remark}
\begin{lemma}\label{lem GOE shift+rank1}
\begin{align*}
&\mathbb{E}\left[\det\left(W_{N-1}-fe_{N-1}e^T_{N-1}+sI_{N-1}\right)\right]\\
&=\mathbb{E}\left[\det\left(W_{N-1}+sI_{N-1}\right)\right]-f\left(\frac{N-2}{N-1}\right)^{\frac{N-2}{2}}\mathbb{E}\left[\det\left(W_{N-2}+\sqrt{\frac{N-1}{N-2}}sI_{N-2}\right)\right]
\end{align*}
\end{lemma}
\begin{proof}
Combine Lemma \ref{lem_det hermite} and \ref{lem char}.
\end{proof}
\begin{theorem}[Plancherel-Rotach asymptotics]\label{THM PR}
	There exists $\delta_0>0$ such that for any $\delta\in \left(0,\delta_0\right)$, we have uniformly in $x\in \left(-\infty,-\sqrt{2}-\delta\right)$, 
	\[\phi_{N}\left(\sqrt{N}x\right)=\left(-1\right)^{N-1}\frac{e^{-NI_1\left(-x\right)}}{\sqrt{4\pi\sqrt{2N}}}h\left(x\right)\left(1+\mathcal{O}\left(N^{-1}\right)\right).\]
	where
	\[h\left(x\right)=\left|\frac{x-\sqrt{2}}{x+\sqrt{2}}\right|^{\frac{1}{4}}+\left|\frac{x+\sqrt{2}}{x-\sqrt{2}}\right|^{\frac{1}{4}},\]
	and
	\[I_1\left(x\right)=\int_{\sqrt{2}}^{x}\sqrt{t^2-2}dt.\]
\end{theorem}
\begin{proof}
	This lemma is the same as Lemma 7.1 in \cite{ABC13} and Lemma 5 in \cite{AB13}.
\end{proof}
From Theorem \ref{THM PR} we derive the following lemma that we need in the proof of Theorem \ref{THM1}.
\begin{lemma}\label{lem_PR}
	There exists $\delta_0>0$ such that for any $\delta\in \left(0,\delta_0\right)$, we have uniformly in $y\in \left(-\infty,-\sqrt{2}-\delta\right)$, 
	\[\phi_{N-1}\left(\sqrt{N}y\right)=\left(-1\right)^{N-2}\frac{e^{-NI_1\left(-y\right)}}{\sqrt{2\pi\sqrt{2N}}}\frac{h\left(y\right)}{\sqrt{y^2-2}-y}\left(1+\mathcal{O}\left(N^{-1}\right)\right).\]
\end{lemma}
\begin{proof}
	Note that $\lim_{N\to \infty}\sqrt{\frac{N}{N-1}}y=y$, so for $N$ large enough and $y\in \left(-\infty,-\sqrt{2}-\delta\right)$, we use Theorem \ref{THM PR} to derive
	\begin{align}
		&\phi_{N-1}\left(\sqrt{N}y\right)=\phi_{N-1}\left(\sqrt{N-1}\frac{\sqrt{N}}{\sqrt{N-1}}y\right)\nonumber\\
		&=\left(-1\right)^{N-2}\frac{e^{-\left(N-1\right)I_1\left(-\frac{\sqrt{N}}{\sqrt{N-1}}y\right)}}{\sqrt{4\pi\sqrt{2N}}}h\left(\frac{\sqrt{N}}{\sqrt{N-1}}y\right)\left(1+\mathcal{O}\left(N^{-1}\right)\right)\nonumber\\
		&=\left(-1\right)^{N-2}\frac{e^{-NI_1\left(-y\right)}e^{I_1\left(-y\right)}e^{-\left(N-1\right)\int_{-y}^{-\sqrt{\frac{N}{N-1}}y}}\sqrt{t^2-2}dt}{\sqrt{4\pi\sqrt{2N}}}h\left(\frac{\sqrt{N}}{\sqrt{N-1}}y\right)\left(1+\mathcal{O}\left(N^{-1}\right)\right)\label{eq_appA_0}
	\end{align}
Since 
\begin{align}
	h\left(\frac{\sqrt{N}}{\sqrt{N-1}}y\right)=h\left(y\right)\left(1+\mathcal{O}\left(N^{-1}\right)\right) \label{eq_appA_1}
\end{align}
and
\begin{align}
	e^{I_1\left(-y\right)}e^{-\left(N-1\right)\int_{-y}^{-\sqrt{\frac{N}{N-1}}y}}\sqrt{t^2-2}dt
&=e^{I_1\left(-y\right)}e^{\frac{y\sqrt{y^2-2}}{2}\left(1+\mathcal{O}\left(N^{-1}\right)\right)}\nonumber\\
&=\frac{\sqrt{2}}{\sqrt{y^2-2}-y}\left(1+\mathcal{O}\left(N^{-1}\right)\right)\label{eq_appA_2}
\end{align}
uniformly for $y\in \left(-\infty,-\sqrt{2}-\delta\right)$, combining Eq.(\ref{eq_appA_0}), (\ref{eq_appA_1}) and (\ref{eq_appA_2}) we prove Lemma \ref{lem_PR}.

\end{proof}

\section{Mathematical Analysis on Thresholds}\label{key}
In this section we discuss the existence and values of $\lambda_{tr}, \lambda^{\left(1\right)}\left(p,k\right)$ and  $\lambda^{\left(2\right)}\left(p,k\right)$.

\begin{lemma}\label{lem.UniqueSol}
If $k\leq 2$, then Eq.$\left(\ref{eq:minima}\right)$ has a unique solution on $\left(0,1\right]$ for any $\lambda>0$. If $k>2$, then Eq.$\left(\ref{eq:minima}\right)$ has a solution if and only if $\lambda\geq \sqrt{p\frac{\left(k-1\right)^{k-1}}{\left(k-2\right)^{k-2}}}$. In particular, when $k>2$ and $\lambda\geq \sqrt{p\frac{\left(k-1\right)^{k-1}}{\left(k-2\right)^{k-2}}}$, the solution on $\left[\sqrt{\frac{k-2}{k-1}},1\right)$ is unique. 
\end{lemma}
\begin{proof}
	When $k=1, 2$, $m_*$ can be solved explicitly from Eq.$\left(\ref{eq:minima}\right)$ as below.
	\begin{equation}\label{eq-m*12}
		m_*\left(\lambda\right)=\begin{cases}
			\sqrt{\frac{\lambda^2}{p}/\left(1+\frac{\lambda^2}{p}\right)}, & k=1 \\
			\sqrt{1-\frac{p}{\lambda^2}}, & k=2.\\
		\end{cases}
	\end{equation}
When $k>2$, let $g\left(m\right)=\frac{\lambda^2m^{2k-4}\left(1-m^2\right)}{p}$. We compute
\begin{equation*}
	g^\prime\left(m\right)=\frac{2\left(k-2\right)\lambda^2}{p}\left(1-\frac{k-1}{k-2}m^2\right)m^{2k-5}.
\end{equation*}
Therefore, Eq. $\left(\ref{eq:minima}\right)$ has a solution on $\left(0,1\right]$ if and only if $g_{\max}=g\left(\sqrt{\frac{k-2}{k-1}}\right)\geq 1$ if and only if $\lambda\geq \sqrt{p\frac{\left(k-1\right)^{k-1}}{\left(k-2\right)^{k-2}}}$. Moreover, when the solution exists on $\left[\sqrt{\frac{k-2}{k-1}},1\right]$, it is unique. 	
\end{proof}
In the next lemma we study the values of $\lambda^{\left(2\right)}\left(p,k\right)$. 
\begin{lemma}\label{lem.highlatitude}
	For any integers $p\geq 3$, $k\geq1$, there exists $\lambda^{\left(2\right)}:=\lambda^{\left(2\right)}\left(p,k\right)>0$ such that $m_*\left(\lambda\right)<m_{\lambda}$ when $\lambda<\lambda^{\left(2\right)}$ and $m_*\left(\lambda\right)\geq m_{\lambda}$ when  $\lambda\geq\lambda^{\left(2\right)}$.
\end{lemma}
\begin{proof}
	When $k=1$, from Eq.$\left(\ref{eq-m*12}\right)$ and Eq. $\left(\ref{eq:lowhighlatitude}\right)$ we derive
	$m_*\left(\lambda\right)\leq m_{\lambda}$ if and only if  $\left(\lambda^2/p\right)^2-\frac{\left(p-2\right)^2}{p-1}\frac{\lambda^2}{p}-\frac{\left(p-2\right)^2}{p-1}<0$ if and only if 
	\begin{align*}
		\lambda\leq \sqrt{\frac{p\left(\frac{\left(p-2\right)^2}{p-1}+\sqrt{\frac{\left(p-2\right)^4}{\left(p-1\right)^2}+\frac{4\left(p-2\right)^2}{p-1}}\right)}{2}}.
	\end{align*}
When $k=2$, again from Eq.$\left(\ref{eq-m*12}\right)$ and Eq. $\left(\ref{eq:lowhighlatitude}\right)$ we derive
$m_*\left(\lambda\right)\leq m_{\lambda}$ if and only if $\frac{p}{\lambda^2}+\frac{p-2}{\sqrt{p-1}}\frac{\sqrt{p}}{\lambda}-1>0$ if and only if
\[\lambda\leq \frac{2\sqrt{p}}{\frac{2-p}{\sqrt{p-1}}+\sqrt{4+\frac{\left(p-2\right)^2}{p-1}}}.\]
When $k>2$, the existence of $\lambda^{\left(2\right)}$ is guaranteed by the fact that $m_*\left(\lambda\right)$ increases to 1 and $m_\lambda$ decreases to 0.
\end{proof}
\begin{lemma}\label{lem.compare}
	For any $p\geq 3$ and $k\geq 1$, $\lambda^{\left(1\right)}\left(p,k\right)=\lambda^{\left(2\right)}\left(p,k\right)$ if and only if $p\leq k$.
\end{lemma}
\begin{proof}
	When $k=1,2$, using Lemma $\ref{lem.UniqueSol}$ and \ref{lem.highlatitude} we have $\lambda^{\left(1\right)}=0<\lambda^{\left(2\right)}$. Therefore from now on we assume $p,k\geq 3$. 
	
	When $p\leq  k$, 
	\begin{align*}
		\left(\frac{m_{\lambda^{\left(1\right)}\left(p,k\right)}}{\sqrt{\frac{k-2}{k-1}}}\right)^{2k}=\frac{\left(p-2\right)^2\left(k-1\right)}{\left(k-2\right)^2\left(p-1\right)}\leq 1.
	\end{align*}
By Lemma \ref{lem.UniqueSol}, $m_{\lambda^{\left(1\right)}\left(p,k\right)}\leq \sqrt{\frac{k-2}{k-1}}\leq m_{*}$ and thus $\lambda^{\left(2\right)}=\lambda^{\left(1\right)}.$

When $p>k$, it suffices to show $g\left(m_{\lambda^{\left(1\right)}\left(p,k\right)}\right)<1 $. A direct computation gives
	\begin{align*}
	g\left(m_{\lambda^{\left(1\right)}\left(p,k\right)}\right)=\frac{\left(p-2\right)^2}{p-1}\cdot \frac{1-m^2_{\lambda^{\left(1\right)}\left(p,k\right)}}{m^4_{\lambda^{\left(1\right)}\left(p,k\right)}}.
\end{align*}
Since $f\left(x\right)=\frac{1-x}{x^2}$ is increasing on $\left[0,1\right]$, and for each fixed $p$, 
\[\log\left(m_{\lambda^{\left(1\right)}\left(p,k\right)}\right)=\frac{1}{k}\left(\log\left(\frac{p-2}{\sqrt{p-1}}\right)-\log\left(\frac{k-2}{\sqrt{k-1}}\right)\right)+\log\left(\frac{k-2}{k-1}\right)\]
is decreasing for $k<p$. 

Therefore, 
\begin{align*}
	g\left(m_{\lambda^{\left(1\right)}\left(p,k\right)}\right)<g\left(m_{\lambda^{\left(1\right)}\left(p,p\right)}\right)=\frac{\left(p-2\right)^2}{p-1}\cdot \frac{1-\frac{p-2}{p-1}}{\left(\frac{p-2}{p-1}\right)^2}=1.
\end{align*}
\end{proof}
\color{black}

\bibliographystyle{acm}
\bibliography{referencePCA}
  
\end{document}